\documentclass[reqno,10pt]{amsart}
\usepackage{amsmath,amssymb,amsthm,mathrsfs}
\usepackage{longtable}
\usepackage{enumerate}
\usepackage{graphicx}
\usepackage{subcaption}
\usepackage{wasysym}
\usepackage{upgreek}
\usepackage[dvipsnames]{xcolor}
\usepackage{cancel}
\usepackage{soul}
\usepackage{mathtools}
\usepackage{tikz}
\usepackage{pgfplots}
\usetikzlibrary{decorations.markings}

\usepackage{comment}
\usepackage{caption}
\allowdisplaybreaks

\numberwithin{equation}{section}

\theoremstyle{plain}
\newtheorem{theorem}{Theorem}[section]
\newtheorem{lemma}[theorem]{Lemma}
\newtheorem{prop}[theorem]{Proposition}

\newtheorem{definition}[theorem]{Definition}
\newtheorem{remark}[theorem]{Remark}

\usepackage{color}

\title{A class of Bernstein-type operators on the unit disk}
\author[M. J. Recarte, M. E. Marriaga, T. E. P\'erez]{M. J. Recarte, M. E. Marriaga, T. E. P\'erez}

\address[M. J. Recarte]{Departamento de Física, Universidad Nacional Autónoma de Honduras en el valle de Sula (Honduras)}
\email{marlon.recarte@unah.edu.hn}

\address[M. E. Marriaga]{Departamento de Matem\'atica Aplicada, Ciencia e Ingenier\'ia de Materiales y
Tecnolog\'ia Electr\'onica, Universidad Rey Juan Carlos (Spain)}
\email{misael.marriaga@urjc.es}

\address[T. E. P\'erez]{Instituto 
de Matem\'aticas IMAG \&
Departamento de Matem\'{a}tica Aplicada, Facultad de Ciencias. Universidad de Granada (Spain)}
\email{tperez@ugr.es}

\thanks{MEM has been supported by Ministerio de Ciencia, Innovaci\'{o}n y Universidades (MICINN) grant PGC2018-096504-B-C33. TEP thanks FEDER/Junta de Andaluc\'ia under grant A-FQM-246-UGR20; MCIN/AEI 10.13039/501100011033 and FEDER funds under grant PGC2018-094932-B-I00; and IMAG-Mar\'ia de Maeztu grant CEX2020-001105-M}

\begin{document}

\begin{abstract}
    We construct and study sequences of linear operators of Bernstein-type acting on bivariate functions defined on the unit disk.  To this end, we study Bernstein-type operators under a domain transformation, we analyse the bivariate Bernstein-Stancu operators, and we introduce Bernstein-type operators on disk quadrants by means of continuously differentiable transformations of the function. We state convergence results for continuous functions and we estimate the rate of convergence. Finally some interesting numerical examples are given, comparing approximations using the shifted Bernstein-Stancu and the Bernstein-type operator on disk quadrants. 
\end{abstract}

\maketitle

\date{\today}

\section{Preliminaries}

In 1912, S. Bernstein (\cite{Be12}) published a constructive proof of the Weierstrass approximation theorem that affirms that every continuous function $f(x)$ defined on a closed interval can be uniformly approximated by polynomials. For a given function $f\in C[0,1]$, Bernstein constructed a sequence of polynomials (lately called Bernstein polynomials) in the form
\begin{equation}\label{B-classic}
B_nf(x)\equiv (B_n\,f)(x) = \sum_{k=0}^n \, f\left(\frac{k}{n}\right)\, \binom{n}{k}\,x^k (1-x)^{n-k}, 
\end{equation}
for $0\leqslant x \leqslant 1$, and $n\geqslant0$.

Clearly, $B_n \,f$ is a polynomial in the variable $x$ of degree less than or equal to $n$, and \eqref{B-classic} can be seen as a linear operator that transforms functions defined on $[0,1]$ to polynomials of degree at most $n$.

Hence, in the sequel, we will refer to $B_n$ as the $n$-th classical univariate \textit{Bernstein operator}.

If we define
\begin{equation}\label{p_nk_classic}
p_{n,k}(x)= \binom{n}{k}\,x^k (1-x)^{n-k}, \quad 0\leqslant x \leqslant 1, \quad 0 \leqslant k \leqslant n,
\end{equation}
then, the set $\{p_{n,k}(x): 0 \leqslant k \leqslant n\}$ is a basis of the linear space of polynomials with real coefficients of degree at most $n$, that we will denote $\Pi_n$, called \textit{Bernstein basis}. Then, the $n$-th Bernstein polynomial associated with $f(x)$ is usually written as
\begin{equation*}
B_n \,f(x) = \sum_{k=0}^n \, f\left(\frac{k}{n}\right)\,p_{n,k}(x).
\end{equation*}

\bigskip

Among others, classical Bernstein operators satisfy the following properties (\cite{L97}):

\begin{itemize}

\item They are linear and positive operators acting on the function $f$, and preserve the constant functions as well as polynomials of degree 1, that is,
$$
B_n \,1 = 1, \quad B_n \, x = x, \quad n\geqslant 0.
$$

\item If $f$ is continuous at a point $x$, then $B_n\,f(x)$ converges to $f(x)$, and $B_n\,f$ converges uniformly if $f$ is continuous on the whole interval $[0,1]$. Moreover, the order of approximation is $\omega_f(n^{-1/2})$, where $\omega_f$ denotes the modulus of continuity of $f$. Because of this property, Bernstein operators are called \textit{Bernstein Approximants}.

\item Bernstein operators satisfy a Voronowskaja type theorem, that is, if $f$ is twice differentiable at $x$, then $B_n\,f(x) - f(x) = \mathcal{O}(1/n)$.

\end{itemize}

The Bernstein operators admit a complete system of polynomial eigenfunctions. However, each eigenfunction depends on $n$ and, thus, is associated with the $n$-th Bernstein operator $B_n$. Another inconvenience of Bernstein operators associated to an adequate function $f$ is its slow rate of convergence towards $f$.

\medskip

For years, several modifications and extensions of Bernstein operators have been studied. The modifications have been introduced in several directions, and we only recall a few interesting cases and cite some papers. For instance, it is possible to substitute the values of the function on equally spaced points by other mean values such as integrals, as was stated in the pioneering papers of Durrmeyer (\cite{Du67}) and Derriennic (\cite{De81}, \cite{De85}). In \cite{CGM06}, the operator is modified in order to preserve some properties of the original function. Another group of modifications given by the transformation of the function by means a convenient continuous and differentiable functions is analysed in \cite{CGR11}; and, of course, the extension of the Bernstein operators to the multivariate case. The most common extension of the Bernstein operator is defined on the unit simplex in higher dimensions (\cite{L97}, \cite{BeScXu92}, \cite{Schurer63}, \cite{Stan63}, among others), since the basic polynomials \eqref{p_nk_classic} can be easily extended to the simplex.

\medskip

In this paper, we are interested in finding an extension of the Bernstein operator to approximate functions defined on the unit disk. In this way, we will need two kinds of modifications: by transformation of the argument of the function to be approximated, and by definition of an adequate basis of functions as \eqref{p_nk_classic}. We present and study two Bernstein-type approximants, and we compare them by means of several examples.

The structure of the paper is as follows. Section \ref{sec1} is devoted to collecting the properties of univariate Bernstein-type operators that we will need along the paper. In Section \ref{sec_BS}, we recall the method introduced by Stancu (\cite{Stan63}) for obtaining Bernstein-type operators in two variables by the successive application of Bernstein operators in one variable. In Section \ref{sec_ShiBS} and Section \ref{sec_SBT}, we define the \textit{shifted $n$-th Bernstein-Stancu operator} and the \textit{shifted $n$-th Bernstein-type operator}, and study their respective approximation properties. The last section is devoted to analyzing several examples, comparing the approximation results for both Bernstein-type operators on the disk.

\bigskip

\section{Univariate Bernstein-type operators}\label{sec1}

In this section, we recall the modified univariate Bernstein-type operators that we will need later. We start by shifting the univariate Bernstein operator.

Using the change of variable 
\begin{equation}\label{change-of-variable}
x=(\beta-\alpha)\,s+\alpha, \quad  \alpha<\beta, \quad 0\leqslant s \leqslant 1,
\end{equation}
the univariate Bernstein basis can be defined on the interval $[\alpha,\beta]$. Indeed, if we let
$$
\widetilde{p}_{n,k}(x;[\alpha,\beta])\,=\,p_{n,k}\left( \frac{x-\alpha}{\beta-\alpha}\right)= \frac{1}{(\beta-\alpha)^n}\binom{n}{k}(x-\alpha)^k(\beta-x)^{n-k},  \quad \alpha\leqslant x \leqslant \beta,
$$
then the set $\{\widetilde{p}_{n,k}(x;[\alpha,\beta]):\,n\geqslant 0, \, 0\leqslant k \leqslant n,\,\alpha\leqslant x \leqslant \beta \}$ is a basis of $\Pi_n$ on the interval $[\alpha,\beta]$ satisfying
\begin{equation*}
	\begin{aligned}
\sum_{k=0}^n \widetilde{p}_{n,k}(x;[\alpha,\beta]) =& \sum_{k=0}^n \,p_{n,k}\left( \frac{x-\alpha}{\beta-\alpha}\right) =
\frac{1}{(\beta-\alpha)^n}\sum_{k=0}^n \binom{n}{k}(x-\alpha)^k(\beta-x)^{n-k} \\ =& \frac{1}{(\beta-\alpha)^n}(x-\alpha + \beta-x)^{n} = 1.
\end{aligned}
\end{equation*}
Moreover, since
$$
\widetilde{p}_{n,k}(x;[\alpha,\beta])\,=\,p_{n,k}(s), \quad 0\leqslant s \leqslant 1, \quad 0\leqslant n, \quad 0\leqslant k \leqslant n,
$$
we have that Bernstein basis on $[\alpha,\beta]$ (see Figure \ref{fbasis})  satisfies the following properties:
\begin{itemize}
    \item $\widetilde{p}_{n,k}(x;[\alpha,\beta])\geqslant 0$ for $\alpha\leqslant x \leqslant \beta$,
    \item $\widetilde{p}_{n,k}(\alpha)=\delta_{0,k}$ and $\widetilde{p}_{n,k}(\beta)=\delta_{k,n}$, where, as usual, $\delta_{\nu,\eta}$ denotes the Kronecker delta,
    \item $(\beta-\alpha)\,\widetilde{p}^{\,\prime}_{n,k}(x;[\alpha,\beta])\,=\,n\left(\widetilde{p}_{n-1,k-1}(x;[\alpha,\beta])-p_{n-1,k}(x;[\alpha,\beta]) \right)$,
    \item If $n\ne 0$, then $\widetilde{p}_{n,k}(x;[\alpha,\beta])$ has a unique local maximum on $[\alpha,\beta]$ at $x=(\beta-\alpha)\,\frac{k}{n}+\alpha$. This maximum takes the value
    $$
    \widetilde{p}_{n,k}\left((\beta-\alpha)\,\frac{k}{n}+\alpha;[\alpha,\beta]\right)\,=\,p_{n,k}\left(\frac{k}{n} \right)= \binom{n}{k} \frac{k^k}{n^n}(n-k)^{n-k}.
    $$
\end{itemize}
\tikzset{
    mark position/.style args={#1(#2)}{
        postaction={
            decorate,
            decoration={
                markings,
                mark=at position #1 with \coordinate (#2);
            }
        }
    }
}

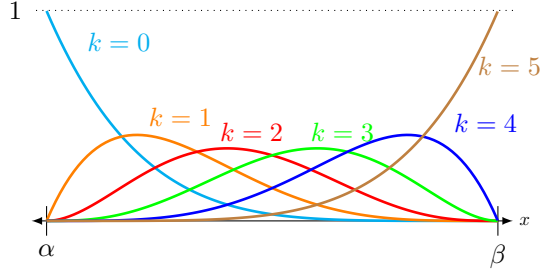
\begin{figure}[h!]
    \centering
\begin{tikzpicture}[scale=2,yscale=1.4]
\pgfmathsetmacro{\alph}{-1}
\pgfmathsetmacro{\bet}{2}
\foreach \x/\xtext in {\alph/\alpha,\bet/\beta}\draw[shift={(\x,0)}] (0pt,2pt) -- (0pt,-2pt) node[below] {$\xtext$};
\draw[latex-latex] ({-0.1+\alph},0)--({\bet+0.1},0)node[scale=0.7,right]{$x$};
\draw[dotted] ({\bet+0.1},1)--({-0.1+\alph},1) node[scale=1,left]{$1$};
\foreach \n in {5}
\foreach \k/\col/\ps in {0/cyan/0.1,1/orange/0.25, 2/red/0.36, 3/green/0.54, 4/blue/0.8, 5/brown/0.85}
{
\draw[line width=1pt,samples=100,color=\col,domain=\alph:\bet,mark position=\ps (a)] plot (\x,{1/(\bet-\alph)^(\n)*(\n!/(\k!*(\n-\k)!)*(\x-\alph)^(\k)*(\bet-\x)^(\n-\k)});
\node[color=\col,above=7pt,right=3pt] at (a) {$k=\k$}
;
}
\end{tikzpicture}
    \caption{Bernstein on basis $[\alpha,\beta]$  for $n=5$.} 
    \label{fbasis}
\end{figure}

For every function $f$ defined on $I=[\alpha,\beta]$, we can define the shifted univariate $n$-th Bernstein operator as
$$
\widetilde{B}_n\,[f(x),I] = \sum_{k=0}^n f\left((\beta-\alpha)\frac{k}{n}+\alpha\right)\,\widetilde{p}_{n,k}(x;I).
$$
Note that $\widetilde{B}_n\,[f(x),I]$ is a polynomial of degree at most $n$. In this way,
$$
\widetilde{B}_n\,[f(x),I] = B_n\,F(s), \quad 0\leqslant s \leqslant 1,
$$
where $F(s)=f((\beta-\alpha)\,s+\alpha)$ is a function defined on $[0,1]$ associated with $f$. From this, and since the change of variable \eqref{change-of-variable} is linear, it is clear that $\widetilde{B}_n$ has analogous properties to those satisfied by the classical Bernstein operator.

\bigskip

In the sequel, we will use the following Bernstein-type operator studied in \cite{CGR11} and \cite{GPR09}:
$$
C_n^{\tau}f  = B_n\,(f\circ \tau^{-1})\circ \tau,
$$
where $\tau$ is any function continuously differentiable as many times as necessary, such that $\tau (0) = 0, \, \tau(1) = 1$, and $\tau'(x) > 0$ for $x \in [0, 1]$. Throughout this work, it will be sufficient for $\tau$ to be continously differentiable.

\medskip

In \cite{CGR11}, the following identities were given:
$$
C_n^{\tau}\,1\,=\,1, \quad C_n^{\tau}\,\tau\,=\,\tau, \quad C_n^{\tau}\,\tau^2\,=\,\left(1-\frac{1}{n}\right)\tau^2+\frac{\tau}{n}.
$$
We have the following result.

\begin{prop}
Let $f$ be a continuous function on $[0,1]$ and $\tau$ is any function that is continuously differentiable, such that $\tau (0) = 0, \, \tau(1) = 1$, and $\tau'(x) > 0$ for $x \in [0, 1]$. Then,
$$
\lim_{n\to \infty} C_n^{\tau}f(x) \, = \, f(x).
$$
That is, $C_n^{\tau}f(x)$ converges uniformly to $f$ on $[0,1]$. 
\end{prop}

\begin{proof}
Set $u = \tau(x)$. We compute
\begin{align*}
 C_n^{\tau}f(x) &= \sum_{k=0}^n f\left( \tau^{-1}\left(\frac{k}{n}\right)\right)\,p_{n,k}(\tau(x)) \\
&= \sum_{k=0}^nf\left( \tau^{-1}\left(\frac{k}{n}\right)\right)\,p_{n,k}(u)\\
&= B_nf\left( \tau^{-1}\left(u\right)\right).
\end{align*}
Since $B_nf\left( \tau^{-1}\left(u\right)\right)\rightarrow f\left( \tau^{-1}\left(u\right)\right)=f(x)$ as $n\rightarrow +\infty$, the result follows from taking the limit on both sides of  $C_n^{\tau}f(x)= B_nf\left( x\right)$.
\end{proof}

\bigskip

We also introduce the following shifted Bernstein-type operator
$$
\widetilde{C}_n^{\tau}[f(x),[\alpha,\beta]] = \sum_{k=0}^nf\circ\tau^{-1}\left((\beta-\alpha)\,\frac{k}{n}+\alpha\right)\, \widetilde{p}_{n,k}(\tau(x);[\alpha,\beta]) , \quad \alpha\leqslant x \leqslant \beta,
$$
where $\tau(x)$ is any function that is continuously differentiable, such that $\tau (\alpha) = \alpha, \, \tau(\beta) = \beta$, and $\tau'(x) > 0$ for $x \in [\alpha, \beta]$.

\begin{prop}
Let $f$ be a continuous function on $[\alpha,\beta]$ and $\tau(x)$ is any function that is continuously differentiable, such that $\tau (\alpha) = \alpha, \, \tau(\beta) = \beta$, and $\tau'(x) > 0$ for $x \in [\alpha, \beta]$. Then,
$$
\lim_{n\to \infty} \widetilde{C}_n^{\tau}[f(x),[\alpha,\beta]] \, = \, f(x).
$$
\end{prop}

\begin{proof}
Set $u = \tau(x)$. We compute
\begin{align*}
 \widetilde{C}_n^{\tau}[f(x),[\alpha,\beta]] &= \sum_{k=0}^n f\left( \tau^{-1}\left((\beta-\alpha)\frac{k}{n}+\alpha\right)\right)\,\widetilde{p}_{n,k}(\tau(x);[\alpha,\beta]) \\
&= \sum_{k=0}^n  f\left( \tau^{-1}\left((\beta-\alpha)\frac{k}{n}+\alpha\right)\right)\,\widetilde{p}_{n,k}(u;[\alpha,\beta])\\
&= \widetilde{B}_n[f\left( \tau^{-1}\left(u\right)\right),[\alpha,\beta]].
\end{align*}
Since $\widetilde{B}_n[f\left( \tau^{-1}\left(u\right)\right),[\alpha,\beta]]\rightarrow f\left( \tau^{-1}\left(u\right)\right)=f(x)$ as $n\rightarrow +\infty$, the result follows from taking the limit on both sides of  $\widetilde{C}_n^{\tau}[f(x),[\alpha,\beta]]= \widetilde{B}_n[f\left( x\right),[\alpha,\beta]]$.
\end{proof}

\bigskip 

\section{Bivariate Bernstein-Stancu operators}\label{sec_BS}

In 1963,  Stancu (\cite{Stan63}) studied a method for deducing polynomials of Bernstein type of two variables. This method is based on obtaining an operator in two variables from the successive application of Bernstein operators of one variable.

Let $\phi_1\equiv\phi_1(x)$ and $\phi_2\equiv \phi_2(x)$ be two continuous functions such that $\phi_1 < \phi_2$ on $[0,1]$. Let $\Omega\subseteq \mathbb{R}^2$ be the domain bounded by the curves $y=\phi_1(x)$, $y=\phi_2(x)$, and the straight lines $x=0$, $x=1$. For every function $f(x,y)$ defined on $\Omega$, define the function
\begin{equation}\label{big-F}
F(x,t)\,=\,f(x,(\phi_2(x)-\phi_1(x))\,t+\phi_1(x)),
\end{equation}
where $0\leqslant t \leqslant 1$.

Notice the change of variable 
\begin{equation}\label{varchange}
y=(\phi_2(x)-\phi_1(x))\,t+\phi_1(x).
\end{equation}

The \textit{$n$-th Bernstein-Stancu operator} is defined as \begin{equation}\label{Bernstein-Stancu}
    \mathscr{B}_n [f(x,y),\Omega]=\sum_{k=0}^n\sum_{j=0}^{n_k}F\left(\frac{k}{n}, \frac{j}{n_k}\right)\,p_{n,k}(x)\,p_{n_k,j}(t),
\end{equation}
where each $n_k$ is a non negative integer associated with the $k$-th node $x_k=k/n$, and $t$ is given by \eqref{varchange}. Writing \eqref{Bernstein-Stancu} explicitly, we have
$$
\mathscr{B}_n [f(x,y),\Omega]= \sum_{k=0}^n\sum_{j=0}^{n_k}F\left(\frac{k}{n}, \frac{j}{n_k}\right)\,p_{n,k}(x)\,p_{n_k,j}\left(\frac{y-\phi_1(x)}{\phi_2(x)-\phi_1(x)}\right), \quad (x,y)\in \Omega.
$$
If we denote by $B_{n}^{(t)}$ the univariate Bernstein operator acting on the variable $t$, then the Bernstein-Stancu operator can be written as
$$
\mathscr{B}_n [f(x,y),\Omega]=\sum_{k=0}^{n}\left[B_{n_k}^{(t)}F\left(\frac{k}{n},t \right)\right]\,p_{n,k}(x).
$$
We have the following representation of $\mathscr{B}_n$ in terms of a matrix determinant.

\begin{prop}
Let $f(x,y)$ be a function defined on the domain $\Omega$, and let $F$ be the function defined on \eqref{big-F}. Denote by $B_{n}^{(t)}$ the univariate Bernstein operator acting on the variable $t$. Then, the $n$-th Bernstein-Stancu operator is given by the determinant
$$
\mathscr{B}_n[f(x,y),\Omega] = -\left|\begin{array}{cccc|c}
1 &   &        & \bigcirc   & B_{n_0}^{(t)}F(0,t)     \\
  & 1 &        &   & B_{n_1}^{(t)}F(1/n,t)   \\
  &   & \ddots &   & \vdots   \\
\bigcirc  &   &        & 1 & B_{n_n}^{(t)}F(1,t)     \\
\hline
p_{n,0}(x) & p_{n,1}(x) & \ldots & p_{n,n}(x) & 0
\end{array} \right|.
$$
\end{prop}

\begin{remark}
Observe that the step size of the partition of the $x$ axis is $1/n$ and, for a fixed node $x_k=k/n$, the step size of the partition of the $t$ axis is $1/n_k$. Therefore, the step size of the partition of the $y$ axis is $1/m_k$, where
$$
m_k= \frac{n_k}{\phi_2\left(\frac{k}{n}\right)-\phi_1\left(\frac{k}{n}\right)},
$$
and, thus,
$$
F\left(\frac{k}{n}, \frac{j}{n_k}\right)= f\left(\frac{k}{n}, \frac{j}{m_k}+ \phi_1\left(\frac{k}{n}\right)\right).
$$
\end{remark}

\bigskip 

We point out that, in general, $\mathscr{B}_n[f(x,y), \Omega]$ is not a polynomial. However, it is possible to obtain polynomials by an appropriate choice of $\phi_1$, $\phi_2$, and $n_k$. For instance:

\noindent (1) \textit{The Bernstein-Stancu operator on the unit square $\mathbf{Q}=[0,1]\times [0,1]$} (see for instance \cite{L97}, \cite{Stan63}) are obtained by letting $\phi_1(x)=0$ and $\phi_2(x)=1$. Hence, for a function $f$ defined on $\mathbf{Q}$, we get 
$$
F\left(\frac{k}{n},\frac{j}{n_k}\right)=f\left(\frac{k}{n},\frac{j}{n_k}\right),
$$ 
and
\begin{equation*}
\mathscr{B}_n[f(x,y),\mathbf{Q}]=\sum_{k=0}^n\sum_{j=0}^{n_k}f\left(\frac{k}{n},\frac{j}{n_k}\right)\,p_{n,k}(x)p_{n_k,j}(y).
\end{equation*}
Note that when $n_k$ is independent of $k$ (e.g., $n_k=m$ for some positive integer $m$), $\mathscr{B}_n$ is the tensor product of univariate Bernstein operators on $\mathbf{Q}$.

\medskip

\noindent (2) \textit{The Bernstein-Stancu operators can be defined on the simplex} $\mathbf{T}^2 = \{(x,y)\in \mathbb{R}: x, y \geqslant0, 1-x-y \geqslant0\}$ (see for instance \cite{BeScXu92} and \cite{Stan63}). In this case, we set $\phi_1(x)=0$, $\phi_2(x)=1-x$, and $n_k=n-k$, $0\leqslant k \leqslant n$. In this way, $m_k=n$ and, since 
$$
F\left(\frac{k}{n},\frac{j}{n-k}\right)=f\left(\frac{k}{n},\frac{j}{n}\right),
$$
we have
\begin{align*}
\mathscr{B}_n[f(x,y),\mathbf{T}^2]&= \sum_{k=0}^n\sum_{j=0}^{n-k}f\left(\frac{k}{n},\frac{j}{n}\right)\, p_{n,k}(x)\,p_{n-k,j}\left(\frac{y}{1-x}\right)\\
&=\sum_{k=0}^n\sum_{j=0}^{n-k}f\left(\frac{k}{n},\frac{j}{n}\right)\,\binom{n}{k\ \ j}\,x^k\,y^j\,(1-x-y)^{n-k-j} , \quad (x,y)\in\mathbf{T}^2,
\end{align*}
where 
$$
\binom{n}{k\ \ j} = \frac{n!}{k!\,j!\,(n-k-j)!},\quad 0\leqslant k+j\leqslant n.
$$

In \cite{Stan63}, Stancu proved the following convergence result on $\mathbf{T}^2$.

\begin{theorem}[\cite{Stan63}]
Let $f$ be a continuous function on $\mathbf{T}^2$. Then $\mathscr{B}_n[f(x,y),\mathbf{T}^2]$ converges uniformly to $f(x,y)$ as $n\to +\infty$.
\end{theorem}

Stancu only gave a detailed proof of the approximation properties of $\mathscr{B}_n$ on triangles. In Section \ref{sec_ShiBS} below, we consider a slightly general operator and prove the uniform convergence on any bounded domain $\Omega$, and we recover Stancu's result when $\Omega= \mathbf{T}^2$.

\bigskip

\section{Bernstein-type operator under a domain transformation }

One way to extend the Bernstein operator on the unit square $\mathbf{Q}$ to another bounded domain $\Omega \in \mathbb{R}^2$ is through an appropriate transformation or change of variables. In this section, we study several cases.

\medskip

\noindent (1) Let $\widehat{\mathbf{Q}}=[-1,1]\times[-1,1]$. The operator defined as
\begin{align*}
\widehat{\mathscr{B}}_n[f(x,y), \widehat{\mathbf{Q}}]=\sum_{k=0}^n\sum_{j=0}^{n_k}f\left(\frac{2k-n}{n},\frac{2j-n_k}{n_k}\right)\,p_{n,k}\left(\frac{x+1}{2}\right)p_{n_k,j}\left(\frac{y+1}{2}\right),
\end{align*}
is a Bernstein operator on $\widehat{\mathbf{Q}}$. Indeed, for every function $f$ defined on $\widehat{\mathbf{Q}}$, we define the function $F: \mathbf{Q}\rightarrow \mathbb{R}$ as
$$
F(u,v)=f(2\,u-1,2\,v-1), \quad (u,v)\in\mathbf{Q}.
$$
Then, using the transformation  $x=2u-1$ and $y=2v-1$ which maps $\mathbf{Q}$ into $\widehat{\mathbf{Q}}$, we get
\begin{align*}
\widehat{\mathscr{B}}_n[f(x,y), \widehat{\mathbf{Q}}]=&\sum_{k=0}^n\sum_{j=0}^{n_k}F\left(\frac{k}{n},\frac{j}{n_k}\right)\,p_{n,k}\left(u\right)p_{n_k,j}\left(v\right)\\
=& \mathscr{B}_n[F(u,v), \mathbf{Q}], \quad (x,y)\in \widehat{\mathbf{Q}}.
\end{align*}

\medskip

\noindent (2) An alternative way to obtain the Bernstein-Stancu operator on the simplex $\mathbf{T}^2$ is by considering the Duffy transformation
\begin{equation*}
x=u,\ y=v(1-u), \quad (u,v)\in \mathbf{Q},
\end{equation*}
which maps $\mathbf{Q}$ into $\mathbf{T}^2$. Let $f$ be a function defined on $\mathbf{T}^2$. We can define the function $F:\mathbf{Q}\rightarrow \mathbb{R}$ as
$$
F(u,v)\,=\,f(u, v\,(1-u)), \quad (u,v)\in \mathbf{Q}.
$$
Then, the operator
\begin{equation*}
\widehat{\mathscr{B}}_n[f(x,y),\mathbf{T}^2]=\sum_{k=0}^n\sum_{j=0}^{n_k}f\left(\frac{k}{n},\frac{j}{n_k}\left(1- \frac{k}{n} \right)\right)\,p_{n,k}(x)\,p_{n_k,j}\left(\frac{y}{1-x}\right),
\end{equation*}
is a Bernstein-type operator on the simplex since, using the Duffy transformation, we get
\begin{align*}
    \widehat{\mathscr{B}}_n[f(x,y),\mathbf{T}^2]= \mathscr{B}_n[F(u,v), \mathbf{Q}], \quad (x,y)\in\mathbf{T}^2.
\end{align*}
Observe that $\widehat{\mathscr{B}}_n[f(x,y),\mathbf{T}^2]$ is not a polynomial unless $n-k-n_k\geqslant 0$. We recover the usual Bernstein-Stancu operator on the simplex by setting $n_k=n-k$.

\medskip

\noindent (3) Consider the unit ball in $\mathbb{R}^2$:
$$
\mathbf{B}^2 \,=\, \{(x,y)\in \mathbb{R}^2:\ x^2+y^2 \leqslant 1 \},
$$
and the transformation $x=2u-1$, $y=(2v-1)\,\sqrt{1-(2u-1)^2}$ which maps the square $\mathbf{Q}$ into $\mathbf{B}^2$. For every function $f$ defined on $\mathbf{B}^2$, we can define the function $F:\mathbf{Q}\rightarrow \mathbb{R}^2$ as
$$
F(u,v)\,=\, f(2u-1, (2v-1)\,\sqrt{1-(2u-1)^2}), \quad (u,v)\in \mathbf{Q}.
$$
The operator
\begin{align*}
&\widehat{\mathscr{B}}_{n}[f(x,y),\mathbf{B}^2]\\
&=\sum_{k=0}^n\sum_{j=0}^{n_k}f\left(\frac{2k-n}{n},\frac{2j-n_k}{n_k}\frac{2\sqrt{k\,(n-k)}}{n}\right)\,p_{n,k}\left(\frac{x+1}{2}\right)p_{n_k,j}\left(\frac{\frac{y}{\sqrt{1-x^2}}+1}{2}\right),
\end{align*}
is a Bernstein operator on the unit ball since 
$$
\widehat{\mathscr{B}}_{n}[f(x,y),\mathbf{B}^2]\,=\,\mathscr{B}_n[F(u,v),\mathbf{Q}], \quad (x,y)\in \mathbf{B}^2.
$$

Observe that, in this case, 
\begin{align*}
p_{n,k}\left(\dfrac{x+1}{2}\right)\,&p_{n_k,j}\left(\dfrac{\frac{y}{\sqrt{1-x^2}}+1}{2}\right)\\
&=\binom{n}{k}\binom{n_k}{j} \frac{(1+x)^k(1-x)^{n-k}\left(\sqrt{1-x^2}+y\right)^j  \left(\sqrt{1-x^2}-y\right)^{n_k-j}}{2^{n+n_k}\sqrt{1-x^2}^{n_k} }.
\end{align*}
In contrast with the previous two cases, there is no obvious choice of $n_k$ such that $\widehat{\mathscr{B}}_{n}[f(x,y),\mathbf{B}^2]$ is a polynomial. Nevertheless, notice that for $y=0$, we have
\begin{align*}
p_{n,k}\left(\dfrac{x+1}{2}\right)\,p_{n_k,j}\left(\dfrac{1}{2}\right)=\frac{1}{2^{n+n_k}}\binom{n}{k}\binom{n_k}{j} (1+x)^k(1-x)^{n-k},
\end{align*}
and for $x=0$ we have
\begin{align*}
p_{n,k}\left(\dfrac{1}{2}\right)\,p_{n_k,j}\left(\dfrac{y+1}{2}\right)=\frac{1}{2^{n+n_k} }\binom{n}{k}\binom{n_k}{j} \left(1+y\right)^j  \left(1-y\right)^{n_k-j}.
\end{align*}
Therefore, $\widehat{\mathscr{B}}_{n}[f(x,y),\mathbf{B}^2]$ is a polynomial on the $x$ and $y$ axes for any choice of $n_k$.

In Figure \ref{meshcircle}, the representation of the mesh in this case for $n=n_k=20$ is given.

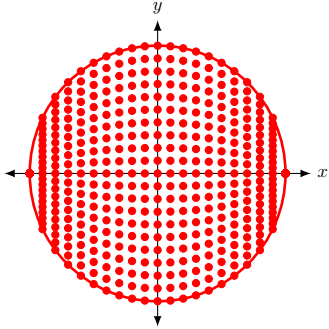
\begin{figure}[h!]

    \centering
\begin{tikzpicture}[scale=1.7]
\draw[red,line width=1pt] (1,0) arc(0:360:1);
\draw[latex-latex] (0,-1.2)--(0,1.2) node[scale=0.7,above]{$y$};
\draw[latex-latex] (-1.2,0)--(1.2,0)node[scale=0.7,right]{$x$};
\foreach \n in {20}
{
\pgfmathsetmacro{\nk}{\n}
\foreach \k in {0,1,2,...,\nk}
{
\foreach \j in {0,1,2,...,\n}
{
\draw[red,fill=red] ({(2*\k-\n)/\n},{2*sqrt(\k*(\n-\k))*(\n-2*\j)/(\n)^2}) circle(0.8pt); 
}
}
}
\end{tikzpicture}
    \caption{Mesh corresponding to case (3) for $n=20$ and $n_k=n=20$ for $0\leqslant k \leqslant n$.} 
    \label{meshcircle}
\end{figure}

\noindent (4) Let 
\begin{eqnarray*}
B_1=\{(x,y)\in \mathbf{B}^2:x\geqslant 0,y\geqslant 0\}, & B_2=\{(x,y)\in \mathbf{B}^2:x\leqslant 0,y\geqslant 0\},\\
B_3=\{(x,y)\in \mathbf{B}^2:x\leqslant 0,y\leqslant 0\}, & B_4=\{(x,y)\in \mathbf{B}^2:x\geqslant 0,y\leqslant 0\},
\end{eqnarray*}
denote the four quadrants of $\mathbf{B}^2$, and consider the transformation 
$$
u=x^2, \quad  v=\dfrac{y^2}{1-x^2}, \quad (x,y)\in \mathbf{B}^2,
$$
which maps each quadrant to $\mathbf{Q}$. The corresponding Bernstein operators on the quadrants are: 
\begin{eqnarray*}
\widehat{\mathscr{B}}_{n}[f(x,y),B_1]&=&\sum_{k=0}^n\sum_{j=0}^{n_k}f\left(\sqrt{\frac{k}{n}},\sqrt{\frac{j(n-k)}{n_k\,n}}\right)\,p_{n,k}(x^2)\,p_{n_k,j}\left(\frac{y^2}{1-x^2}\right), \\
\widehat{\mathscr{B}}_{n}[f(x,y),B_2]&=&\sum_{k=0}^n\sum_{j=0}^{n_k}f\left(-\sqrt{\frac{k}{n}},\sqrt{\frac{j(n-k)}{n_k\,n}}\right)\,p_{n,k}(x^2)\,p_{n_k,j}\left(\frac{y^2}{1-x^2}\right),\\
\widehat{\mathscr{B}}_{n}[f(x,y),B_3]&=&\sum_{k=0}^n\sum_{j=0}^{n_k}f\left(-\sqrt{\frac{k}{n}},-\sqrt{\frac{j(n-k)}{n_k\,n}}\right)\,p_{n,k}(x^2)\,p_{n_k,j}\left(\frac{y^2}{1-x^2}\right),\\
\widehat{\mathscr{B}}_{n}[f(x,y),B_4]&=&\sum_{k=0}^n\sum_{j=0}^{n_k}f\left(\sqrt{\frac{k}{n}},-\sqrt{\frac{j(n-k)}{n_k\,n}}\right)\,p_{n,k}(x^2)\,p_{n_k,j}\left(\frac{y^2}{1-x^2}\right).
\end{eqnarray*}
Indeed, for every function $f$ defined on $\mathbf{B}^2$, we can define the functions on $\mathbf{Q}$:
\begin{align*}
    &F_1(u,v)  =  f(\sqrt{u}, \sqrt{v\,(1-u)}), & F_2(u,v)= f(-\sqrt{u},\sqrt{v\,(1-u)}),\\
    &F_3(u,v) = f(-\sqrt{u}, -\sqrt{v\,(1-u)}), & F_4(u,v) = f(\sqrt{u},-\sqrt{v\,(1-u)}).
\end{align*}
Then,
\begin{eqnarray*}
&\widehat{\mathscr{B}}_{n}[f(x,y),B_1]= \mathscr{B}_n[F_1(u,v),\mathbf{Q}], &\widehat{\mathscr{B}}_{n}[f(x,y),B_2]=\mathscr{B}_n[F_2(u,v),\mathbf{Q}],\\
&\widehat{\mathscr{B}}_{n}[f(x,y),B_3]=\mathscr{B}_n[F_3(u,v),\mathbf{Q}], &\widehat{\mathscr{B}}_{n}[f(x,y),B_4]=\mathscr{B}_n[F_4(u,v),\mathbf{Q}].
\end{eqnarray*}
If we choose $n_k=n-k$, we have that $\widehat{\mathscr{B}}_{n}[f(x,y),B_i]$, $i=1,2,3,4$, are polynomials of degree $2n$ since
\begin{equation*}
p_{n,k}(x^2)\,p_{n-k,j}\left(\frac{y^2}{1-x^2}\right)\,=\,\binom{n}{k\ \ j}x^{2k}y^{2j}(1-x^2-y^2)^{n-k-j}.
\end{equation*}
In this case, observe that for $k=0$, the mesh corresponding to $B_1$ and $B_2$, and similarly to $B_3$ and $B_4$, coincide on the $y$ axis (see Figure \ref{circular_mesh}). Moreover, for $j=0$, the mesh corresponding to adjacent quadrants coincide on the $x$ axis. Therefore, we can define a \textit{piece-wise} Bernstein operator on $\mathbf{B}^2$ as follows:
\begin{equation}\label{piecewise}
\overline{\mathscr{B}}_n[f(x,y), \mathbf{B}^2]\,=\,\left\{ \begin{array}{cc}
\widehat{\mathscr{B}}_{n}[f(x,y),B_1],     &  (x,y)\in B_1, \\
\widehat{\mathscr{B}}_{n}[f(x,y),B_2],     &  (x,y)\in B_2, \\
\widehat{\mathscr{B}}_{n}[f(x,y),B_3],     &  (x,y)\in B_3, \\
\widehat{\mathscr{B}}_{n}[f(x,y),B_4],     &  (x,y)\in B_4.
\end{array}
\right.
\end{equation}

\begin{figure}[h]
    \centering
\begin{tikzpicture}[scale=1.7]
\draw[red,line width=1pt] (1,0) arc(0:360:1);
\draw[latex-latex] (0,-1.2)--(0,1.2) node[scale=0.7,above]{$y$};
\draw[latex-latex] (-1.2,0)--(1.2,0)node[scale=0.7,right]{$x$};
\foreach \n in {10}
{
\foreach \k in {0,1,2,...,\n}
{
\pgfmathsetmacro{\nk}{\n-\k}
\foreach \j in {0,...,\nk}
{
\draw[cyan,fill=cyan] ({sqrt(\k/\n)},{sqrt(\j/\n)}) circle(0.8pt);
\draw[orange,fill=orange] ({sqrt(\k/\n)},{-sqrt(1-\k/\n-\j/\n)}) circle(0.8pt);
}
\foreach \j in {0,...,\k}
{
\draw[green,fill=green] ({-sqrt(1-\k/\n)},{sqrt(\j/\n)}) circle(0.8pt);
\draw[blue,fill=blue] ({-sqrt(1-\k/\n)},{-sqrt(\k/\n-\j/\n)}) circle(0.8pt);
}
}
}
\end{tikzpicture}
    \caption{Circular mesh after applying the transformation $(u,v)\mapsto (\sqrt{u}, \sqrt{v\,(1-u)}$ for $(u,v)\in \mathbf{Q}$ with $n=10$ and $n_k=n-k$, for $0\leqslant k \leqslant n$.} 
    \label{circular_mesh}
\end{figure}
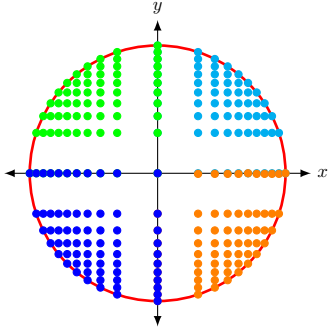

\begin{prop}\label{continuity-1}
For any function $f$ on $\mathbf{B}^2$,  $\overline{\mathscr{B}}_n[f(x,y), \mathbf{B}^2]$ is a continuous function on $\mathbf{B}^2$.
\end{prop}
\begin{proof}
Clearly, $\overline{\mathscr{B}}_n[f(x,y), \mathbf{B}^2]$ is continuous on the interior of each quadrant.

For $x=0$,
\begin{align*}
\displaystyle\left.\widehat{\mathscr{B}}_{n}[f(x,y),B_1]\right|_{x=0}=&\displaystyle\sum_{k=0}^n\sum_{j=0}^{n_k}f\left(\sqrt{\frac{k}{n}},\sqrt{\frac{j(n-k)}{n_k\,n}}\right)\,p_{n,k}(0)\,p_{n_k,j}\left(y^2\right), \\
=&\displaystyle\sum_{j=0}^{n_0}f\left(0,\sqrt{\frac{j}{n_0}}\right)\,p_{n_0,j}\left(y^2\right)\\
=&\displaystyle\left.\widehat{\mathscr{B}}_{n}[f(x,y),B_2]\right|_{x=0},
\end{align*}
and
\begin{align*}
\displaystyle\left.\widehat{\mathscr{B}}_{n}[f(x,y),B_3]\right|_{x=0}=&\displaystyle\sum_{k=0}^n\sum_{j=0}^{n_k}f\left(-\sqrt{\frac{k}{n}},-\sqrt{\frac{j(n-k)}{n_k\,n}}\right)\,p_{n,k}(0)\,p_{n_k,j}\left(y^2\right), \\
=&\displaystyle\sum_{j=0}^{n_0}f\left(0,-\sqrt{\frac{j}{n_0}}\right)\,p_{n_0,j}\left(y^2\right)\\
=&\displaystyle\left.\widehat{\mathscr{B}}_{n}[f(x,y),B_4]\right|_{x=0}.
\end{align*}

Similarly, for $y=0$
$$
\left.\widehat{\mathscr{B}}_{n}[f(x,y),B_1]\right|_{y=0}= \left.\widehat{\mathscr{B}}_{n}[f(x,y),B_4]\right|_{y=0},
$$
and
$$
\left.\widehat{\mathscr{B}}_{n}[f(x,y),B_2]\right|_{y=0}= \left.\widehat{\mathscr{B}}_{n}[f(x,y),B_3]\right|_{y=0}.
$$
Therefore, $\overline{\mathscr{B}}_n[f(x,y), \mathbf{B}^2]$ is continuous on the $x$ and $y$ axes.
\end{proof}

\bigskip

\section{Shifted Bernstein-Stancu operators}\label{sec_ShiBS}

Motivated by the examples of Bernstein operators on different domains introduced in the previous section, now we define the \textit{shifted $n$-th Bernstein-Stancu operator} and study its approximation properties.

Let $\phi_1$ and $\phi_2$ be two continuous functions, and let $I=[a,b]$ be an interval such that $\phi_1 < \phi_2$ on $I$. Let $\Omega\subset \mathbb{R}^2$ be the domain bounded by the curves $y=\phi_1(x)$, $y=\phi_2(x)$, and the straight lines $x=a$, $x=b$. Observe that for a fixed $x\in I$, the polynomials $\widetilde{p}_{n,k}(y;[\phi_1(x),\phi_2(x)])$, $n\geqslant 0$, $0\leqslant k \leqslant n$, constitute a univariate shifted Bernstein basis on the interval $[\phi_1(x),\phi_2(x)]$.

For every function $f(x,y)$ defined on $\Omega$, define the function
\begin{equation}\label{F-tilde}
\widetilde{F}(u,v;\Omega)\,=\,f\left((b-a)\,u+a,(\widetilde{\phi}_2(u)-\widetilde{\phi}_1(u))\,v+\widetilde{\phi}_1(u)\right),
\end{equation}
where 
$$
\widetilde{\phi}_i(u) = \phi_i((b-a)\,u+a), \quad i=1,2,
$$
$0\leqslant u \leqslant 1$, and $0\leqslant v \leqslant 1$. 

The shifted $n$-th Bernstein-Stancu operator is defined as
$$
\widetilde{\mathscr{B}}_n[f(x,y),\Omega]=\sum_{k=0}^n\sum_{j=0}^{n_k}\widetilde{F}\left(\frac{k}{n},\frac{j}{n_k};\Omega \right)\,\widetilde{p}_{n,k}(x;I)\,\widetilde{p}_{n_k,j}(y;[\phi_1,\phi_2]), \quad (x,y)\in \Omega,
$$
where $n_k=n-k$ or $n=k$ or all $0\leqslant k \leqslant n$. Written in terms of the univariate Bernstein basis, we get
$$
\widetilde{\mathscr{B}}_n[f(x,y),\Omega]=\sum_{k=0}^n\sum_{j=0}^{n_k}\widetilde{F}\left(\frac{k}{n},\frac{j}{n_k};\Omega\right)\,p_{n,k}\left( \frac{x-a}{b-a}\right)\,p_{n_k,j}\left(\frac{y-\phi_{1}(x)}{\phi_2(x)-\phi_1(x)}\right).
$$

The following result plays an important role when studying the convergence of the shifted Bernstein-Stancu operator.

\begin{lemma}\label{monomials}
Let $\phi_1$ and $\phi_2$ be two continuous functions, and let $I=[a,b]$ be an interval such that $\phi_1 < \phi_2$ on $I$. Let $\Omega\subset \mathbb{R}^2$ be the domain bounded by the curves $y=\phi_1(x)$, $y=\phi_2(x)$, and the straight lines $x=a$, $x=b$. Then:

\begin{enumerate}
    \item[(i)] $\widetilde{\mathscr{B}}_n[1,\Omega]=1$,
    \item[(ii)] $\widetilde{\mathscr{B}}_n[x,\Omega]=x$,
    \item[(iii)] $\widetilde{\mathscr{B}}_n[y,\Omega]\to y$ as $n\to +\infty$ uniformly on $[a,b]$,
    \item[(iv)] $\widetilde{\mathscr{B}}_n[x^2,\Omega]=x^2+(x-a)\,(b-x)/n$,
    \item[(v)] $\widetilde{\mathscr{B}}_n[xy,\Omega]\to x\,y$ as $n\to +\infty$ uniformly on $[a,b]$,
    \item[(vi)] $\widetilde{\mathscr{B}}_n[y^2,\Omega]\to y^2$ as $n\to +\infty$ uniformly on $[a,b]$. 
\end{enumerate}
\end{lemma}

\begin{proof}
(i) Obviously $\widetilde{\mathscr{B}}_n[1,\Omega] = 1$. 

\noindent
(ii) We compute
\begin{align*}
    \widetilde{\mathscr{B}}_n[x,\Omega]=& \sum_{k=0}^n \sum_{j=0}^{n_k}\widetilde{p}_{n,k}(x;I)\,\widetilde{p}_{n_k,j}(y;[\phi_1,\phi_2])\,\left((b-a)\frac{k}{n}+a\right)\\
     =& (b-a)\sum_{k=0}^n\widetilde{p}_{n,k}(x;I)\frac{k}{n}\left(\sum_{j=0}^{n_k}\widetilde{p}_{n_k,j}(y;[\phi_1,\phi_2]) \right)+a\,\widetilde{\mathscr{B}}_n[1,\Omega]\\
    =&(b-a)\sum_{k=0}^n\widetilde{p}_{n,k}(x;I)\frac{k}{n}+a\\
    =&(b-a)\frac{x-a}{b-a}\sum_{k=0}^{n-1}\widetilde{p}_{n-1,k}(x;I)+a\\
    =& x. 
\end{align*}

\noindent
(iii) Observe that 
\begin{equation}\label{monomio_iii}
\begin{aligned}
    \sum_{j=0}^{n_k}\widetilde{p}_{n_k,j}&(y;[\phi_1,\phi_2])\,\frac{j}{n_k} 
    =  \sum_{j=1}^{n_k}\binom{n_k}{j}\, \frac{(y-\phi_1(x))^j\,(\phi_2(x)-y)^{n_k-j}}{(\phi_2(x)-\phi_1(x))^{n_k}}\,\frac{j}{n_k}\\
    = & \sum_{j=0}^{n_k-1}\binom{n_k-1}{j}\,\frac{(y-\phi_1(x))^{j+1}\,(\phi_2(x)-y)^{n_k-1-j}}
    {(\phi_2(x)-\phi_1(x))^{n_k}}\\
    = & \frac{y - \phi_1(x)}{\phi_2(x)-\phi_1(x)}.
    \end{aligned}
\end{equation}
Therefore, applying the linearity, we get
\begin{align*}
    \widetilde{\mathscr{B}}_n[y,\Omega]  = & \sum_{k=0}^n \sum_{j=0}^{n_k}\widetilde{p}_{n,k}(x;I)\,
    \widetilde{p}_{n_k,j}(y;[\phi_1,\phi_2])\,\left[\left(\widetilde{\phi}_2\left(\frac{k}{n}\right) - \widetilde{\phi}_1\left(\frac{k}{n}\right)\right) \frac{j}{n_k} + \widetilde{\phi}_1\left(\frac{k}{n}\right)\right]\\
   = & \left[\sum_{k=0}^n \left(\widetilde{\phi}_2\left(\frac{k}{n}\right) - \widetilde{\phi}_1\left(\frac{k}{n}\right)\right)\widetilde{p}_{n,k}(x;I)\right]\frac{y - \phi_1(x)}{\phi_2(x)-\phi_1(x)} \\
    &+ \sum_{k=0}^n \widetilde{\phi}_1\left(\frac{k}{n}\right)\widetilde{p}_{n,k}(x;I)\\
    = & \widetilde{B}_n[\phi_2 - \phi_1,I] \,\frac{y - \phi_1(x)}{\phi_2(x)-\phi_1(x)} + \widetilde{B}_n [\phi_1,I],
    \end{align*}
where $\widetilde{B}_n$ denotes the univariate shifted Bernstein operator acting on the variable $x$. Since $\widetilde{B}_n$ converges uniformly for a continuous function, we have 
\begin{align*}
     \lim_{n\to+\infty} \widetilde{\mathscr{B}}_n[y,\Omega]=& \,\lim_{n\to +\infty}\widetilde{B}_n[\phi_2 - \phi_1,I] \,\frac{y - \phi_1(x)}{\phi_2(x)-\phi_1(x)}+\lim_{n\to +\infty} \widetilde{B}_n [\phi_1,I]\\
     =&\,[\phi_2(x) - \phi_1(x)]\,\frac{y - \phi_1(x)}{\phi_2(x)-\phi_1(x)}  + \phi_1(x)\\
     =&\,y.
\end{align*}

\noindent
(iv) We compute
\begin{align*}
    \widetilde{\mathscr{B}}_n[x^2,\Omega] = & \sum_{k=0}^n\sum_{j=0}^{n_k}\widetilde{p}_{n,k}(x;I)\,\widetilde{p}_{n_k,j}(y;[\phi_1,\phi_2])\,\left((b-a)\frac{k}{n}+a \right)^2\\
    =& (b-a)^2\sum_{k=0}^{n}\widetilde{p}_{n,k}(x;I)\frac{k^2}{n^2}+2\,a\,(b-a)\frac{x-a}{b-a}\sum_{k=0}^{n-1}\widetilde{p}_{n-1,k}(x;I)+a^2\\
    =& (b-a)^2\left(\frac{n-1}{n}\left(\frac{x-a}{b-a} \right)^{\!2}\,\sum_{k=0}^{n-2}\widetilde{p}_{n-2,k}(x;I)+\frac{1}{n}\frac{x-a}{b-a}\sum_{k=0}^{n-1}\widetilde{p}_{n-1,k}(x;I) \right)\\
    &+2\,a\,(x-a)+a^2\\
    =&x^2+\frac{(x-a)(b-x)}{n}.
\end{align*}

\noindent (v) Taking $f(x,y)=x\,y$ in \eqref{F-tilde}, we have
\begin{align*}
    \widetilde{F}\left(\frac{k}{n},\frac{j}{n_k};\Omega\right)=&\left((b-a)\,\frac{k}{n}+a \right)\left(\left(\widetilde{\phi}_2\left(\frac{k}{n}\right)-\widetilde{\phi}_1\left(\frac{k}{n}\right)\right)\,\frac{j}{n_k}+\widetilde{\phi}_1\left(\frac{k}{n}\right) \right)\\
    =&\left((b-a)\,\frac{k}{n}+a \right)\,\left(\widetilde{\phi}_2\left(\frac{k}{n}\right)-\widetilde{\phi}_1\left(\frac{k}{n}\right)\right)\,\frac{j}{n_k}\\
    &+\left((b-a)\,\frac{k}{n}+a \right)\,\widetilde{\phi}_1\left(\frac{k}{n}\right).
\end{align*}
Then, from \eqref{monomio_iii} and the linearity of $\widetilde{\mathscr{B}}_n$, we get
$$
\widetilde{\mathscr{B}}_n[xy,\Omega]\,=\,\widetilde{B}_n[x\left(\phi_2(x)-\phi_1(x) \right),I]\,\frac{y-\phi_1(x)}{\phi_2(x)-\phi_1(x)}+\widetilde{B}_n[x\,\phi_1(x),I].
$$
Since $\widetilde{B}_n$ converges uniformly for a continuous function, $\widetilde{\mathscr{B}}_n[xy,\Omega](x,y)\to x\,y$ as $n\to +\infty$ uniformly.

\noindent
(vi) Finally, if $f(x,y)=y^2$ in \eqref{F-tilde}, then we get
\begin{equation*}
    \begin{aligned}
    \widetilde{F}\left(\frac{k}{n},\frac{j}{n_k};\Omega\right)=&\left[\left(\widetilde{\phi}_2\left(\frac{k}{n}\right) - \widetilde{\phi}_1\left(\frac{k}{n}\right)\right) \frac{j}{n_k} + \widetilde{\phi}_1\left(\frac{k}{n}\right)\right]^2\\
    =&\left(\widetilde{\phi}_2\left(\frac{k}{n}\right) - \widetilde{\phi}_1\left(\frac{k}{n}\right)\right)^2 \frac{j^2}{n_k^2}\\
    &+2\left(\widetilde{\phi}_2\left(\frac{k}{n}\right) - \widetilde{\phi}_1\left(\frac{k}{n}\right)\right)\,\widetilde{\phi}_1\left(\frac{k}{n}\right)\, \frac{j}{n_k}+\widetilde{\phi}_1\left(\frac{k}{n}\right)^2.
    \end{aligned}
\end{equation*}
    Then,
\begin{align*}
    \widetilde{\mathscr{B}}_n[y^2,\Omega] = & \sum_{k=0}^n\sum_{j=0}^{n_k}\widetilde{p}_{n,k}(x;I)\,\widetilde{p}_{n_k,j}(y;[\phi_1,\phi_2])\,\left(\widetilde{\phi}_2\left(\frac{k}{n}\right) - \widetilde{\phi}_1\left(\frac{k}{n}\right)\right)^2 \frac{j^2}{n_k^2}\\
    &+2\sum_{k=0}^n\sum_{j=0}^{n_k}\widetilde{p}_{n,k}(x;I)\widetilde{p}_{n_k,j}(y;[\phi_1,\phi_2])\left(\widetilde{\phi}_2\left(\frac{k}{n}\right) - \widetilde{\phi}_1\left(\frac{k}{n}\right)\right)\widetilde{\phi}_1\left(\frac{k}{n}\right) \frac{j}{n_k}\\
    &+    \sum_{k=0}^n\sum_{j=0}^{n_k}\widetilde{p}_{n,k}(x;I)\,\widetilde{p}_{n_k,j}(y;[\phi_1,\phi_2])\,\widetilde{\phi}_1\left(\frac{k}{n}\right)^2.
    \end{align*}

Observe that
\begin{align*}
    \sum_{j=0}^{n_k}\widetilde{p}_{n_k,j}(y;[\phi_1,\phi_2])\, \frac{j^2}{n_k^2}=& \frac{n_k-1}{n_k}\left(\frac{y-\phi_1(x)}{\phi_2(x)-\phi_1(x)} \right)^2\sum_{j=0}^{n_k-2}\widetilde{p}_{n_k-2,j}(y;[\phi_1,\phi_2])\\
    &+\frac{1}{n_k}\frac{y-\phi_1(x)}{\phi_2(x)-\phi_1(x)}\sum_{j=0}^{n_k-1}\widetilde{p}_{n_k-1,j}(y;[\phi_1,\phi_2])\\
    =&\frac{n_k-1}{n_k}\left(\frac{y-\phi_1(x)}{\phi_2(x)-\phi_1(x)} \right)^2+\frac{1}{n_k}\frac{y-\phi_1(x)}{\phi_2(x)-\phi_1(x)}\\
    =&\left(\frac{y-\phi_1(x)}{\phi_2(x)-\phi_1(x)} \right)^2+\frac{(y-\phi_1(x))\,(\phi_2(x)-y)}{n_k\,(\phi_2(x)-\phi_1(x))^2}.
    \end{align*}
    Together with \eqref{monomio_iii}, we get
    \begin{align*}
    \widetilde{\mathscr{B}}_n[y^2,\Omega] = & \left(\frac{y-\phi_1(x)}{\phi_2(x)-\phi_1(x)} \right)^2\widetilde{B}_n[\phi_2-\phi_1,I]^2\\
    &+\frac{(y-\phi_1(x))\,(\phi_2(x)-y)}{n\,(\phi_2(x)-\phi_1(x))^2}\,\sum_{k=0}^n\widetilde{p}_{n,k}(x;I)\,\left(\widetilde{\phi}_2\left(\frac{k}{n}\right) - \widetilde{\phi}_1\left(\frac{k}{n}\right)\right)^2\frac{1}{n_k/n}\\
    &+2\left(\frac{y-\phi_1(x)}{\phi_2(x)-\phi_1(x)} \right)\,\widetilde{B}_n[(\phi_2-\phi_1)\,\phi_1,I]+    \widetilde{B}_n[\phi_1^2,I].
    \end{align*}
    
    If $n_k=n-k$, then
    $$
    \sum_{k=0}^n\widetilde{p}_{n,k}(x;I)\,\left(\widetilde{\phi}_2\left(\frac{k}{n}\right) - \widetilde{\phi}_1\left(\frac{k}{n}\right)\right)^2\frac{1}{n_k/n}\,=\,\widetilde{B}_n\left[\frac{(\phi_2(x)-\phi_1(x))^2}{1-\frac{x-a}{(b-a)}} ,I\right],
    $$
    and if $n_k=k$, then
    $$
    \sum_{k=0}^n\widetilde{p}_{n,k}(x;I)\,\left(\widetilde{\phi}_2\left(\frac{k}{n}\right) - \widetilde{\phi}_1\left(\frac{k}{n}\right)\right)^2\frac{1}{n_k/n}\,=\,\widetilde{B}_n\left[\frac{(\phi_2(x)-\phi_1(x))^2}{\frac{x-a}{(b-a)}} ,I\right].
    $$
    In either case, $\widetilde{\mathscr{B}}_n[y^2,\Omega]\to y^2$ as $n\to +\infty$.
\end{proof}

The convergence of the operator is clear from Lemma \ref{monomials} and Korovkin's theorem (\cite{L97}).

Now, we study the approximation properties of the shifted Bernstein-Stancu operators.

\begin{definition}[\cite{Schurer63}]
Let $f$ be a function defined on $\Omega$. The modulus of continuity of $f$ is defined by
$$
\omega(\delta_1,\delta_2) \,=\,\sup |f(x'',y'')-f(x',y')|,
$$
where $\delta_1,\delta_2>0$ are real numbers, whereas $(x',y')$ and $(x'',y'')$ are points of $\Omega$ such that $|x''-x'|\leqslant \delta_1$ and $|y''-y'|\leqslant  \delta_2$.
\end{definition}

\begin{theorem}\label{approx_Stancu}
Let $f$ be a continuous function on $\Omega$.  Then,
$$
\lim_{n\to+\infty}\widetilde{\mathscr{B}}_n[f(x,y),\Omega] = f(x,y),
$$
uniformly on $\Omega$.
\end{theorem}

\begin{proof}

Let $\delta_1,\delta_2>0$ be real numbers.

Note that on $\Omega$ we have $\widetilde{\mathscr{B}}_n[1,\Omega]=1$,
$$
\widetilde{p}_{n,k}(x;I)\,\widetilde{p}_{n_k,j}(y;[\phi_1,\phi_2])\geqslant 0, \quad 0\leqslant k \leqslant n, \quad 0\leqslant j\leqslant n_k,
$$
and
$$
|f(x'',y'')-f(x',y')|\leqslant \omega(|x''-x'|,|y''-y'|)\leqslant w(\delta_1,\delta_2).
$$
Taking into account the inequality (see, for instance, \cite{Schurer63, Stan63})
$$
\omega(c_1\,\delta_1,c_2\,\delta_2)\leqslant (c_1+c_2+1)\,\omega(\delta_1,\delta_2), \quad c_1,c_2>0,
$$
we compute
\begin{align*}
    &\left|f(x,y)-\widetilde{F}\left(\frac{k}{n},\frac{j}{n_k};\Omega \right) \right|\\
    &\leqslant \omega\left(\left|x-(b-a)\frac{k}{n}-a\right|, \left|y-\left(\widetilde{\phi}_2\left(\frac{k}{n}\right)-\widetilde{\phi}_1\left(\frac{k}{n}\right)\right)\,\frac{j}{n_k}-\widetilde{\phi}_1\left(\frac{k}{n}\right) \right| \right)\\
    &\leqslant (\lambda_1+\lambda_2+1)\,\omega(\delta_1,\delta_2),
\end{align*}
where
$$
\lambda_1 \equiv \lambda_1(x,n,k,\delta_1, a, b)= \frac{1}{\delta_1}\left|x-(b-a)\frac{k}{n}-a\right|,
$$
and 
$$
\lambda_2\equiv \lambda_2(x,n,k,n_k,\delta_2, \phi_1, \phi_2) = \frac{1}{\delta_2}\left|y-\left(\widetilde{\phi}_2\left(\frac{k}{n}\right)-\widetilde{\phi}_1\left(\frac{k}{n}\right)\right)\,\frac{j}{n_k}-\widetilde{\phi}_1\left(\frac{k}{n}\right) \right|.
$$
Therefore,
\begin{align*}
    &|f(x,y)-\widetilde{\mathscr{B}}_n[f(x,y),\Omega]|\\
    &\leqslant \sum_{k=0}^n\sum_{j=0}^{n_k}\widetilde{p}_{n,k}(x;I)\,\widetilde{p}_{n_k,j}(y;[\phi_1,\phi_2])\left|f(x,y)-\widetilde{F}\left(\frac{k}{n},\frac{j}{n_k};\Omega \right) \right|\\
    &\leqslant\sum_{k=0}^n\sum_{j=0}^{n_k}\widetilde{p}_{n,k}(x;I)\,\widetilde{p}_{n_k,j}(y;[\phi_1,\phi_2])(\lambda_1+\lambda_2+1)\,\omega(\delta_1,\delta_2).
\end{align*}
We will deal with each term in the last inequality separately.

Since $\mathscr{B}[1,\Omega]=1$, $0\leqslant \widetilde{p}_{n,k}(x;I)\leqslant 1$, $0\leqslant \widetilde{p}_{n_k,j}(y;[\phi_1,\phi_2])\leqslant 1$, and $x\mapsto x^{1/2}$ is a concave function, by Jensen's inequality, we have
\begin{align*}
	&\sum_{k=0}^n\sum_{j=0}^{n_k}\widetilde{p}_{n,k}(x;I)\,\widetilde{p}_{n_k,j}(y;[\phi_1,\phi_2])\left|x-(b-a)\frac{k}{n}-a\right|\\
	& = \sum_{k=0}^n\sum_{j=0}^{n_k}\widetilde{p}_{n,k}(x;I)\,\widetilde{p}_{n_k,j}(y;[\phi_1,\phi_2])\left[\left(x-(b-a)\frac{k}{n}-a\right)^2\right]^{1/2}\\
	&\leqslant \left[\sum_{k=0}^n\sum_{j=0}^{n_k}\widetilde{p}_{n,k}(x;I)\,\widetilde{p}_{n_k,j}(y;[\phi_1,\phi_2])\left(x-(b-a)\frac{k}{n}-a\right)^2 \right]^{1/2}.
\end{align*}
Using (i), (ii), and (iv) in Lemma \ref{monomials}, we get
\begin{align*}
    &\sum_{k=0}^n\sum_{j=0}^{n_k}\widetilde{p}_{n,k}(x;I)\,\widetilde{p}_{n_k,j}(y;[\phi_1,\phi_2])\left|x-(b-a)\frac{k}{n}-a\right|\\
    &= \left[x^2\,\widetilde{\mathscr{B}}_n[1,\Omega]-2\,x\,\widetilde{\mathscr{B}}_n[x,\Omega]+\widetilde{\mathscr{B}}_n[x^2,\Omega] \right]^{1/2}\to 0 \quad \text{as} \quad n\to +\infty,
\end{align*}
uniformly since  $\widetilde{\mathscr{B}}_n[1,\Omega]=1$, $\widetilde{\mathscr{B}}_n[x,\Omega]=x$, and $\lim_{n\to+\infty}\widetilde{\mathscr{B}}_n[x^2,\Omega]=x^2$.

Similarly, from Jensen's inequality, and using (i), (iii), and (vi) in Lemma \ref{monomials}, we get
\begin{align*}
    &\sum_{k=0}^n\sum_{j=0}^{n_k}\widetilde{p}_{n,k}(x;I)\,\widetilde{p}_{n_k,j}(y;\phi_1,\phi_2)\left|y-\left(\widetilde{\phi}_2\left(\frac{k}{n}\right)-\widetilde{\phi}_1\left(\frac{k}{n}\right)\right)\,\frac{j}{n_k}-\widetilde{\phi}_1\left(\frac{k}{n}\right) \right|\\
    &\leqslant \left[\sum_{k=0}^n\sum_{j=0}^{n_k}\widetilde{p}_{n,k}(x;I)\,\widetilde{p}_{n_k,j}(y;[\phi_1,\phi_2])\left(y-\left(\widetilde{\phi}_2\left(\frac{k}{n}\right)-\widetilde{\phi}_1\left(\frac{k}{n}\right)\right)\,\frac{j}{n_k}-\widetilde{\phi}_1\left(\frac{k}{n}\right) \right)^2 \right]^{1/2}\\
    &= \left[y^2\,\widetilde{\mathscr{B}}_n[1,\Omega]-2\,y\,\widetilde{\mathscr{B}}_n[y,\Omega]+\widetilde{\mathscr{B}}_n[y^2,\Omega] \right]^{1/2}\to 0 \quad \text{as} \quad n\to +\infty,
\end{align*}
uniformly since  $\widetilde{\mathscr{B}}_n[1,\Omega]=1$, $\lim_{n\to+\infty}\widetilde{\mathscr{B}}_n[y,\Omega]= y$, and $\lim_{n\to+\infty}\widetilde{\mathscr{B}}_n[y^2,\Omega]=y^2$.

Finally, choosing $\delta_1=\delta_2=1/\sqrt{n}$, then $\omega(1/\sqrt{n},1/\sqrt{n})\to 0$ as $n\to +\infty$, and, thus, $\widetilde{\mathscr{B}}_n[f(x,y),\Omega]$ converges uniformly to $f(x,y)$ on $\Omega$.
\end{proof}

Recall that the univariate shifted Bernstein satisfy the following Voronowskaja type asymptotic formula: Let $f(x)$ be bounded on the interval $I$, and let $x_0\in I$ at which $f''(x_0)$ exists. Then,
\begin{equation}\label{Vo}
    \left.\widetilde{B}_n[f(x),I]\right|_{x=x_0}-f(x_0)=\mathcal{O}(n^{-1}).
\end{equation}
Now, we give an analogous result for the Bernstein-Stancu operator.

\begin{theorem}
Let $f(x,y)$ be a bounded function on $\Omega=\{(x,y)\in \mathbb{R}^2:\ a\leqslant x \leqslant b,\ \phi_1(x)\leqslant y \leqslant \phi_2(x)\}$, and let $(x_0,y_0)\in \Omega$ be a point at which $f(x,y)$ admits second order partial derivatives, and $\phi_i''(x_0)$, $i=1,2$, exist. Then,
$$
\left.\widetilde{\mathscr{B}}_n[f(x,y), \Omega]\right|_{(x,y)=(x_0,y_0)}-f(x_0,y_0)=\mathcal{O}\left(\frac{1}{n} \right).
$$
\end{theorem}
\begin{proof}
Let us write the Taylor expansion of $f(u,v)$ at the point $(x_0,y_0)$:
\begin{align*}
f(u,v)=&f(x_0,y_0)+(u-x_0)\,f_x(x_0,y_0)+(v-y_0)\,f_y(x_0,y_0)+\frac{(x-x_0)^2}{2}f_{xx}(x_0,y_0)\\[5pt]
&+\frac{(u-x_0)(v-y_0)}{2}\,(f_{xy}(x_0,y_0)+f_{yx}(x_0,y_0))+\frac{(v-y_0)^2}{2}\,f_{yy}(x_0,y_0)\\[5pt]
&+||(u,v)-(x_0,y_0)||^2\,h(u,v),
\end{align*}
where $h(u,v)$ is a bounded function such that $h(u,v)\rightarrow 0$ as $(u,v)\rightarrow (x_0,y_0)$. Applying $\widetilde{\mathscr{B}}_n$ to both sides, we get:
\begin{align*}
&\widetilde{\mathscr{B}}_n[f(u,v)]=f(x_0,y_0)+f_x(x_0,y_0)\,\widetilde{\mathscr{B}}_n[u-x_0]+\widetilde{\mathscr{B}}_n[v-y_0]\,f_y(x_0,y_0)\\[5pt]
&+\frac{1}{2}f_{xx}(x_0,y_0)\,\widetilde{\mathscr{B}}_n[(u-x_0)^2]+\frac{1}{2}\,(f_{xy}(x_0,y_0)+f_{yx}(x_0,y_0))\,\widetilde{\mathscr{B}}_n[(u-x_0)(v-y_0)]\\[5pt]
&+\frac{1}{2}\,f_{yy}(x_0,y_0)\,\widetilde{\mathscr{B}}_n[(v-y_0)^2]+\widetilde{\mathscr{B}}_n\left[||(u,v)-(x_0,y_0)||^2\,h(u,v)\right],
\end{align*}
where we have omitted $\Omega$ for brevity. We deal with each term separately.

From Lemma \ref{monomials} (ii), we get $\left.\widetilde{\mathscr{B}}_n[u-x_0]\right|_{u=x_0}=0$. Next, from the proof of Lemma \ref{monomials} (iii), we have
$$
\widetilde{\mathscr{B}}_n[v-y_0]=\widetilde{B}_n[\phi_2(x) - \phi_1(x),I] \,\frac{y - \phi_1(x)}{\phi_2(x)-\phi_1(x)} + \widetilde{B}_n [\phi_1(x),I]-y_0.
$$
But using \eqref{Vo}, we get
\begin{align*}
\left.\widetilde{\mathscr{B}}_n[v-y_0]\right|_{(u,v)=(x_0,y_0)}=&\left(\phi_2(x_0) - \phi_1(x_0)+\mathcal{O}\left(\frac{1}{n}\right) \right)\,\frac{y_0 - \phi_1(x_0)}{\phi_2(x_0)-\phi_1(x_0)} \\
&+ \phi_1(x_0)+\mathcal{O}\left(\frac{1}{n}\right)-y_0\,=\,\mathcal{O}\left(\frac{1}{n}\right) .
\end{align*}
Similarly,
\begin{equation*}
\begin{array}{l}
    \left.\widetilde{\mathscr{B}}_n[(u-x_0)^2]\right|_{(u,v)=(x_0,y_0)}\,=\,\mathcal{O}\left(\dfrac{1}{n}\right), \\[15pt] \left.\widetilde{\mathscr{B}}_n[(u-x_0)(v-y_0)]\right|_{(u,v)=(x_0,y_0)}\,=\,\mathcal{O}\left(\dfrac{1}{n}\right),\\[15pt]
    \left.\widetilde{\mathscr{B}}_n[(v-y_0)^2]\right|_{(u,v)=(x_0,y_0)}\,=\,\mathcal{O}\left(\dfrac{1}{n}\right).
    \end{array}
\end{equation*}

Now we deal with the last term
\begin{align*}
    &\widetilde{\mathscr{B}}_n\left[||(u,v)-(x_0,y_0)||^2\,h(u,v)\right]\\
    &=\sum_{k=0}^n\sum_{j=0}^{n_k}\widetilde{F}\left(\frac{k}{n},\frac{j}{n_k}\right)\,\widetilde{H}\left(\frac{k}{n},\frac{j}{n_k}\right)\widetilde{p}_{n,k}(x;I)\,\widetilde{p}_{n_k,j}(y;[\phi_1,\phi_2]),
\end{align*}
where 
\begin{align*}
\widetilde{F}&\left(\frac{k}{n},\frac{j}{n_k}\right)\\
&=\left((b-a)\frac{k}{n}+a-x_0 \right)^2+\left(\left(\widetilde{\phi}_2\left(\frac{k}{n}\right) -\widetilde{\phi}_1\left(\frac{k}{n}\right)\right)\frac{j}{n_k}+\widetilde{\phi}_1\left(\frac{k}{n}\right)-y_0 \right)^2,
\end{align*}
and
$$
\widetilde{H}\left(\frac{k}{n},\frac{j}{n_k}\right)=h\left((b-a)\frac{k}{n}+a,\left(\widetilde{\phi}_2\left(\frac{k}{n}\right) -\widetilde{\phi}_1\left(\frac{k}{n}\right)\right)\frac{j}{n_k}+\widetilde{\phi}_1\left(\frac{k}{n}\right) \right).
$$

Fix a real number $\varepsilon >0$. Then there is a real number $\delta>0$ such that if $||(u,v)-(x_0,y_0)||<\delta$, then $|h(u,v)|<\varepsilon$. Let $S_{\delta}$ be the set of $k$ and $j$ such that $\frac{1}{\delta^2}\widetilde{F}\left(\frac{k}{n},\frac{j}{n_k}\right)>1$. Then,
\begin{align*}
&\sum_{(k,j)\in S_{\delta}}\widetilde{p}_{n,k}(x_0;I)\,\widetilde{p}_{n_k,j}(y_0;[\phi_1,\phi_2])\\
&<\frac{1}{\delta^2}\sum_{(k,j)\in S_{\delta}}\widetilde{F}\left(\frac{k}{n},\frac{j}{n_k}\right)\widetilde{p}_{n,k}(x_0;I)\,\widetilde{p}_{n_k,j}(y_0;[\phi_1,\phi_2])\\
&\leqslant \frac{1}{\delta^2}\left.\left(\widetilde{\mathscr{B}}_n[(u-x_0)^2]+\widetilde{\mathscr{B}}_n[(v-y_0)^2]\right)\right|_{(u,v)=(x_0,y_0)}\\
&=\mathcal{O}\left(\frac{1}{n}\right).
\end{align*}
Moreover, we have
\begin{align*}
   & \sum_{(k,j)\not \in S_{\delta}}\widetilde{F}\left(\frac{k}{n},\frac{j}{n_k}\right)\,\left|\widetilde{H}\left(\frac{k}{n},\frac{j}{n_k}\right)\right|\widetilde{p}_{n,k}(x;I)\,\widetilde{p}_{n_k,j}(y;[\phi_1,\phi_2])\\
   & <\varepsilon \sum_{(k,j)\not \in S_{\delta}}\widetilde{F}\left(\frac{k}{n},\frac{j}{n_k}\right)\widetilde{p}_{n,k}(x_0;I)\,\widetilde{p}_{n_k,j}(y_0;[\phi_1,\phi_2])\leqslant \mathcal{O}\left(\frac{1}{n}\right).
\end{align*}
Thus,
\begin{align*}
    &\left|\widetilde{\mathscr{B}}_n\left[||(u,v)-(x_0,y_0)||^2\,h(u,v)\right]\right|\\
    &\leqslant \sum_{(k,j)\in S_{\delta}}\left|\widetilde{F}\left(\frac{k}{n},\frac{j}{n_k}\right)\,\widetilde{H}\left(\frac{k}{n},\frac{j}{n_k}\right)\right|\widetilde{p}_{n,k}(x;I)\,\widetilde{p}_{n_k,j}(y;[\phi_1,\phi_2])\\
    & +\sum_{(k,j)\not \in S_{\delta}}\widetilde{F}\left(\frac{k}{n},\frac{j}{n_k}\right)\,\left|\widetilde{H}\left(\frac{k}{n},\frac{j}{n_k}\right)\right|\widetilde{p}_{n,k}(x;I)\,\widetilde{p}_{n_k,j}(y;[\phi_1,\phi_2])\\
    &\leqslant M\sum_{(k,j)\in S_{\delta}}\widetilde{p}_{n,k}(x_0;I)\,\widetilde{p}_{n_k,j}(y_0;[\phi_1,\phi_2])+\mathcal{O}\left(\frac{1}{n}\right)\\
    &\leqslant \mathcal{O}\left(\frac{1}{n}\right),
\end{align*}
where 
$$
M\,=\,\sup_{(u,v)\in \Omega}\left|||(u,v)-(x_0,y_0)||^2\,h(u,v)\right|.
$$

Putting all the above together, we get
$$
\left|\left.\widetilde{\mathscr{B}}_n[f(x,y), \Omega]\right|_{(x,y)=(x_0,y_0)}-f(x_0,y_0)\right|\leqslant \mathcal{O}\left(\frac{1}{n} \right),
$$
and the result follows.
\end{proof}

\section{Shifted Bernstein-type operators}\label{sec_SBT}

We define the \textit{shifted} bivariate Bernstein-type operator. Let $\phi_1$ and $\phi_2$ be two continuous functions, and let $I=[a,b]$ be an interval such that $\phi_1 < \phi_2$ on $I$. Let $\Omega\subset \mathbb{R}^2$ be the domain bounded by the curves $y=\phi_1(x)$, $y=\phi_2(x)$, and the straight lines $x=a$, $x=b$. Let 
$$
T(x,y)\,=\, (\tau(x), \sigma_x(y)), \quad (x,y)\in \Omega,
$$
where $\tau$ is any continuously differentiable function on $I$, such that $\tau (a) = a, \, \tau(b) = b$, and $\tau'(x) > 0$ for $x \in I$, and for each fixed $x\in I$, $\sigma_x$ is any continuously differentiable function on $[\phi_1(x),\phi_2(x)]$, such that $\sigma_x ( \phi_1(x)) = \phi_1(x), \, \sigma_x(\phi_2(x)) = \phi_2(x)$, and $\sigma_x'(y) > 0$ for $y \in [\phi_1(x), \phi_2(x)]$.

For every function $f(x,y)$ defined on $\Omega$, define the function
\begin{equation*}
\widetilde{F}^{T}(u,v; \Omega)\,=\,f\circ T^{-1} \left((b-a)\,u+a,(\widetilde{\phi}_2(u)-\widetilde{\phi}_1(u))\,v+\widetilde{\phi}_1(u)\right),
\end{equation*}  
for $0\leqslant u \leqslant 1$, and $0\leqslant v \leqslant 1$, where $\widetilde{\phi}_i$, $i=1,2$, are defined in \eqref{F-tilde}. 

\emph{The shifted bivariate Bernstein-type operator} is defined as
$$
\widetilde{\mathscr{C}}_{n}^{T}[f(x,y),\Omega]=\sum_{k=0}^n\sum_{j=0}^{n_k}\widetilde{F}^{T}\left(\frac{k}{n},\frac{j}{n_k};\Omega \right)\,\widetilde{p}_{n,k}(\tau(x);I)\,\widetilde{p}_{n_k,j}(\sigma_x(y);[\phi_1,\phi_2]), 
$$
for $(x,y)\in \Omega$, where $n_k=n-k$ or $n_k=k$ for $0\leqslant k \leqslant n$.

Written in terms of the univariate classical Bernstein basis, we get
$$
\widetilde{\mathscr{C}}_{n}^{T}[f(x,y),\Omega]=\sum_{k=0}^n\sum_{j=0}^{n_k}\widetilde{F}^{T}\left(\frac{k}{n},\frac{j}{n_k};\Omega \right)\,p_{n,k}\left( \frac{\tau(x)-a}{b-a}\right)\,p_{n_k,j}\left(\frac{\sigma_x(y)-\phi_{1}(x)}{\phi_2(x)-\phi_1(x)}\right).
$$

\begin{prop}
For every function $f(x,y)$ defined on $\Omega$,
$$
\lim_{n\to +\infty}\widetilde{\mathscr{C}}_{n}^{T}[f(x,y),\Omega]\,=\,f(x,y).
$$
\end{prop}
\begin{proof}
Let $u=\tau(x)$ and, for each $x\in I$, $v=\sigma_x(y)$. Then,
\begin{align*}
    \widetilde{\mathscr{C}}_{n}^{T}[f(x,y),\Omega]&=\sum_{k=0}^n\sum_{j=0}^{n_k}\widetilde{F}\left(\frac{k}{n},\frac{j}{n_k};a,b \right)\, \widetilde{p}_{n,k}(u;I)\,\widetilde{p}_{n_k,j}(v;[\phi_1,\phi_2])\\
    &=\widetilde{\mathscr{B}}_n[(f\circ T)(u,v),\Omega].
\end{align*}
From Theorem \ref{approx_Stancu}, we have $\widetilde{\mathscr{B}}_n[(f\circ T)(u,v),\Omega]=\widetilde{\mathscr{B}}_n[f(x,y),\Omega]$ converges uniformly to $f(x,y)$. Hence, $\widetilde{\mathscr{C}}_{n}^{T}[f(x,y),\Omega]$ converges uniformly to $f(x,y)$.
\end{proof}

\bigskip

Now, we study shifted Bernstein-type operators defined on each quadrant of $\mathbf{B}^2$, denoted by $B_i$ for $i=1,2,3,4$. We will choose $T$ and $n_k$ such that, for any function, the approximation given by Bernstein-type operators on each quadrant is a polynomial.

\medskip

\noindent
(i) For $x\in [0,1]$, let $\tau(x)=x^2$ and, for each fixed value of $x$, let $\sigma_x(y)=y^2/\sqrt{1-x^2}$. Let $n_k=n-k$, $\phi_1(x)=0$, and $\phi_2(x)=\sqrt{1-x^2}$. Then,
$$
\widetilde{p}_{n,k}(x^2;[0,1])=\binom{n}{k}x^{2\,k}\,(1-x^2)^{n-k},
$$
$$
\widetilde{p}_{n-k,j}(\sigma_x(y);[\phi_1,\phi_2])=\frac{1}{(1-x^2)^{n-k}}\binom{n-k}{j}y^{2\,j}\,(1-x^2-y^2)^{n-k-j},
$$
$$
\widetilde{F}^{T}(u,v;B_1)\,=\,f\left(\sqrt{u},\,\sqrt{(1-u)\,v}\right),
$$
where $B_1=\{(x,y)\in \mathbb{R}^2:\, x^2+y^2\leqslant 1,\, x,y\geqslant 0 \}$. Then,
\begin{equation*}
\widetilde{\mathscr{C}}_{n}^{T}[f(x,y),B_1]=\sum_{k=0}^n\sum_{j=0}^{n-k}f\left(\sqrt{\frac{k}{n}},\,\sqrt{\frac{j}{n}}\right)\,\binom{n}{k\ \ j}x^{2k}y^{2j}(1-x^2-y^2)^{n-k-j}.
\end{equation*}

\bigskip

\noindent
(ii) For $x\in [-1,0]$, let $\tau(x) = -x^2$ and, for each fixed value of $x$, let $\sigma_x(y)=y^2/\sqrt{1-x^2}$. Let $n_k=k$, $\phi_1(x)=0$, and $\phi_2(x)=\sqrt{1-x^2}$. Then,
$$
\widetilde{p}_{n,k}(-x^2;[-1,0])=\binom{n}{k}(1-x^2)^k\,x^{2n-2k},
$$
$$
\widetilde{p}_{k,j}(\sigma_x(y);[\phi_1,\phi_2])=\frac{1}{(1-x^2)^{k}}\binom{k}{j}y^{2j}\left(1-x^2-y^2\right)^{k-j},
$$
$$
\widetilde{F}^{T}(u,v;B_2)\,=\,f(-\sqrt{1-u},\,\sqrt{u\,v}),
$$
where $B_2=\{(x,y)\in \mathbb{R}^2:\, x^2+y^2\leqslant 1,\, x\leqslant 0, \ y\geqslant 0 \}$. Then,
\begin{equation*}
\widetilde{\mathscr{C}}_{n}^{T}[f(x,y),B_2]=\sum_{k=0}^n\sum_{j=0}^{k}f\left(-\sqrt{1-\frac{k}{n}},\,\sqrt{\frac{j}{n}}\right)\,\binom{n}{k}\binom{k}{j}x^{2n-2k}y^{2j}(1-x^2-y^2)^{k-j}.
\end{equation*}

\bigskip

\noindent \textit{(iii)} For $x\in [-1,0]$, let $\tau(x) = -x^2$ and, for each fixed value of $x$, let $\sigma_x(y)=-y^2/\sqrt{1-x^2}$. Let $n_k=k$, $\phi_1(x)=-\sqrt{1-x^2}$, and $\phi_2(x)=0$. Then,
$$
\widetilde{p}_{n,k}(-x^2;[-1,0])=\binom{n}{k}(1-x^2)^k\,x^{2n-2k},
$$
$$
\widetilde{p}_{k,j}(\sigma_x(y);[\phi_1,\phi_2])=\frac{1}{(1-x^2)^k}\binom{k}{j}\left(1-x^2-y^2\right)^jy^{2k-2j},
$$
$$
\widetilde{F}^{T}(u,v;B_3)\,=\,f\left(-\sqrt{1-u},-\sqrt{u\,(1-v)}\right),
$$
where $B_3=\{(x,y)\in \mathbb{R}^2:\, x^2+y^2\leqslant 1,\, x,y\leqslant 0\}$. Then,
\begin{equation*}
\widetilde{\mathscr{C}}_{n}^{T}[f(x,y),B_3]=\sum_{k=0}^n\sum_{j=0}^{k}f\left(-\sqrt{1-\frac{k}{n}},\,-\sqrt{\frac{k-j}{n}}\right)\,\binom{n}{k}\binom{k}{j}x^{2n-2k}y^{2k-2j}(1-x^2-y^2)^{j}.
\end{equation*}

\bigskip

\noindent 
(iv) For $x\in [0,1]$, let $\tau(x) = x^2$ and, for each fixed value of $x$, let $\sigma_x(y)=-y^2/\sqrt{1-x^2}$. Let $n_k=n-k$, $\phi_1(x)=-\sqrt{1-x^2}$, and $\phi_2(x)=0$. Then,
$$
\widetilde{p}_{n,k}(x^2;[0,1])=\binom{n}{k}x^{2\,k}\,(1-x^2)^{n-k},
$$
$$
\widetilde{p}_{n-k,j}(\sigma_x(y);[\phi_1,\phi_2])=\frac{1}{(1-x^2)^{n-k}}\binom{n-k}{j}(1-x^2-y^2)^j\,y^{2n-2k-2j},
$$
$$
\widetilde{F}^{T}(u,v;B_4)\,=\,f\left(\sqrt{u},\,-\sqrt{(1-u)\,(1-v)}\right),
$$
where $B_4=\{(x,y)\in \mathbb{R}^2:\, x^2+y^2\leqslant 1,\, x\geqslant 0, \ y\leqslant 0 \}$. Then,
\begin{equation*}
\widetilde{\mathscr{C}}_{n}^{T}[f(x,y),B_4]=\sum_{k=0}^n\sum_{j=0}^{n-k}f\left(\sqrt{\frac{k}{n}},\,-\sqrt{1-\frac{k+j}{n}}\right)\,\binom{n}{k\ \ j}x^{2k}(1-x^2-y^2)^j\,y^{2n-2k-2j}.
\end{equation*}

Similar to \eqref{piecewise}, we can define a piece-wise Bernstein-type operator on $\mathbf{B}^2$ as follows:
\begin{equation}\label{piecewise2}
    \overline{\mathscr{C}}_n[f(x,y), \mathbf{B}^2]\,=\,\left\{ \begin{array}{cc}
\widetilde{\mathscr{C}}^{T}_{n}[f(x,y),B_1],     &  (x,y)\in B_1, \\
\widetilde{\mathscr{C}}^{T}_{n}[f(x,y),B_2],     &  (x,y)\in B_2, \\
\widetilde{\mathscr{C}}^{T}_{n}[f(x,y),B_3],     &  (x,y)\in B_3, \\
\widetilde{\mathscr{C}}^{T}_{n}[f(x,y),B_4],     &  (x,y)\in B_4.
\end{array}
\right.
\end{equation}
The proof of the following proposition is similar to that of Proposition \ref{continuity-1}.

\begin{prop}
For any function $f$ on $\mathbf{B}^2$,  $\overline{\mathscr{C}}_n[f(x,y), \mathbf{B}^2]$ is a continuous function on $\mathbf{B}^2$.
\end{prop}

\section{Numerical experiments}\label{secNumExp}

In this section, we present numerical experiments where we compare the shifted Bernstein-Stancu operator $\widetilde{\mathscr{B}}_n$ on $\mathbf{B}^2$,  and the shifted Bernstein-type operator $\overline{\mathscr{C}}_{n}$ in \eqref{piecewise2}. To do this, we consider different functions defined on $\mathbf{B}^2$. For each function $f(x,y)$, we compute $\widetilde{\mathscr{B}}[f(x,y),\mathbf{B}^2]$ and $\overline{\mathscr{C}}_{n}[f(x,y),\mathbf{B}^2]$. For $\widetilde{\mathscr{B}}_n[f(x,y),\mathbf{B}^2]$, we get $(n+1)^2$ mesh points $(x_i,y_i)$. We set $z_i=f(x_i,y_i)$, $1\leqslant i\leqslant (n+1)^2$ and $\hat{z}_i$ equal to the value of $\widetilde{\mathscr{B}}_n[f(x,y),\mathbf{B}^2]$ at the respective mesh point, and compute the root mean square error (RMSE) as follows:
$$
\operatorname{RMSE}(f,\widetilde{\mathscr{B}}_n)\,=\,\sqrt{\sum_{i=1}^{(n+1)^2}\frac{(z_i-\hat{z}_i)^2}{(n+1)^2}}.
$$
Similarly, for $\overline{\mathscr{C}}_n[f(x,y),\mathbf{B}^2]$, we get $2\,n\,(n+1)$ mesh points $(\bar{x}_j,\bar{y}_j)$. We set $w_j=f(\bar{x}_j,\bar{y}_j)$, $1\leqslant j\leqslant 2\,n\,(n+1)$ and $\bar{w}_j$ equal to the value of $\overline{\mathscr{C}}_n[f(x,y),\mathbf{B}^2]$ at the respective mesh point, and compute the RMSE as follows:
$$
\operatorname{RMSE}(f,\overline{\mathscr{C}}_n)\,=\,\sqrt{\sum_{j=1}^{2\,n\,(n+1)}\frac{(w_j-\bar{w}_j)^2}{2\,n\,(n+1)}}.
$$
In each case, we plot the RSME for increasing values of $n$ using Mathematica.

For $\overline{\mathscr{C}}_n[f(x,y), \mathbf{B}^2]$, we represent the approximation on each quadrant using different colors as shown in Figure \ref{numeexp1}. We take $n = 100$, then the mesh for each quadrant consists $20200$ points.

For $\widetilde{\mathscr{B}}_n[f(x,y),\mathbf{B}^2]$, we take $n=200$. Then the mesh for all the unit disk consists of 40401 points. 

We note that the operator $\overline{\mathscr{C}}_n$ requires two evaluations at the mesh points on the common boundaries of two adjacent quadrants. Therefore, the operator $\widetilde{\mathscr{B}}_n$ needs a smaller number of evaluations than the operator $\overline{\mathscr{C}}_n$ since, for a fixed $n$, $\widetilde{\mathscr{B}}_n$ and $\overline{\mathscr{C}}_n$ are composed of $(n+1)^2$ and $2\,(n+1)\,(n+2)$ evaluations,respectively.  

\begin{center}
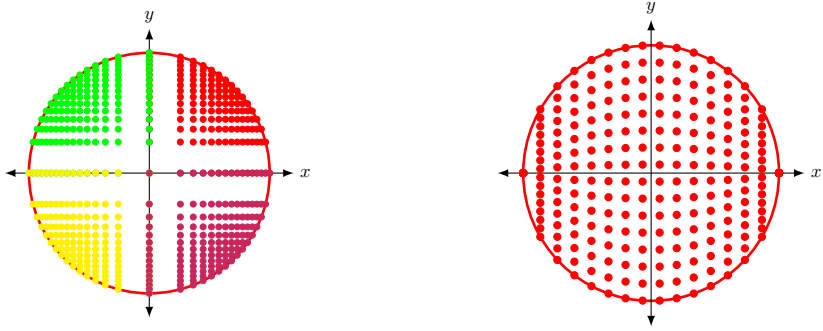
\begin{figure}[h!]
\begin{tabular}{cc}
\begin{minipage}{0.49\textwidth}
\begin{tikzpicture}[scale=1.6]
\draw[red,line width=1pt] (1,0) arc(0:360:1);
\draw[latex-latex] (0,-1.2)--(0,1.2) node[scale=0.7,above]{$y$};
\draw[latex-latex] (-1.2,0)--(1.2,0)node[scale=0.7,right]{$x$};
\foreach \n in {15}
{
\foreach \k in {0,1,2,...,\n}
{
\pgfmathsetmacro{\nk}{\n-\k}
\foreach \j in {0,...,\nk}
{
\draw[red,fill=red] ({sqrt(\k/\n)},{sqrt(\j/(\n))}) circle(0.7pt) ; 
\draw[green,fill=green] ({-sqrt(\k/\n)},{sqrt(\j/(\n))}) circle(0.7pt); 
\draw[yellow,fill=yellow] ({-sqrt(\k/\n)},{-sqrt(\j/(\n))}) circle(0.7pt); 
\draw[purple!85,fill=purple!85] ({sqrt(\k/\n)},{-sqrt(\j/(\n))}) circle(0.7pt); 
}
}
}
\end{tikzpicture}
\end{minipage}&
\begin{minipage}{0.49\textwidth}
\begin{tikzpicture}[scale=1.7]
\draw[red,line width=1pt] (1,0) arc(0:360:1);
\draw[latex-latex] (0,-1.2)--(0,1.2) node[scale=0.7,above]{$y$};
\draw[latex-latex] (-1.2,0)--(1.2,0)node[scale=0.7,right]{$x$};
\foreach \n in {15}
{
\pgfmathsetmacro{\nk}{\n}
\foreach \k in {0,1,2,...,\nk}
{
\foreach \j in {0,1,2,...,\n}
{
\draw[red,fill=red] ({(2*\k-\n)/\n},{2*sqrt(\k*(\n-\k))*(\n-2*\j)/(\n)^2}) circle(0.8pt); 
}
}
}
\end{tikzpicture}
\end{minipage}
\end{tabular}
\caption{Left: Mesh for $\overline{\mathscr{C}}_n$ with $n=15$. Color code for disk quadrants ($B_1$ red; $B_2$ green; $B_3$ yellow; $B_4$ purple). Right: Mesh for $\widetilde{\mathscr{B}}_n$ with $n=15$. } 
    \label{numeexp1}
\end{figure}
\end{center}

\subsection{Example 1}

First, we consider the continuous function 
$$
f(x,y)=x\sin(5x-6y)+y, \quad (x,y)\in \mathbf{B}^2.
$$
The graph of $f(x,y)$ is shown in Figure \ref{triga}, and the approximations $\overline{\mathscr{C}}_{n}[f(x,y),\mathbf{B}^2]$  and $\widetilde{\mathscr{B}}[f(x,y),\mathbf{B}^2]$  are shown in Figure \ref{aprox1}. We list the RSME of both approximations for different values of $n$ in Table \ref{trigaterror} and plot them together in Figure \ref{trigaperror}, where the characteristic slow convergence inherited from the univariate Bernstein operators is observed.\\

\begin{figure}[h!]
\centering
\includegraphics[scale=0.2,trim=5cm 15cm 8cm 10cm, clip]{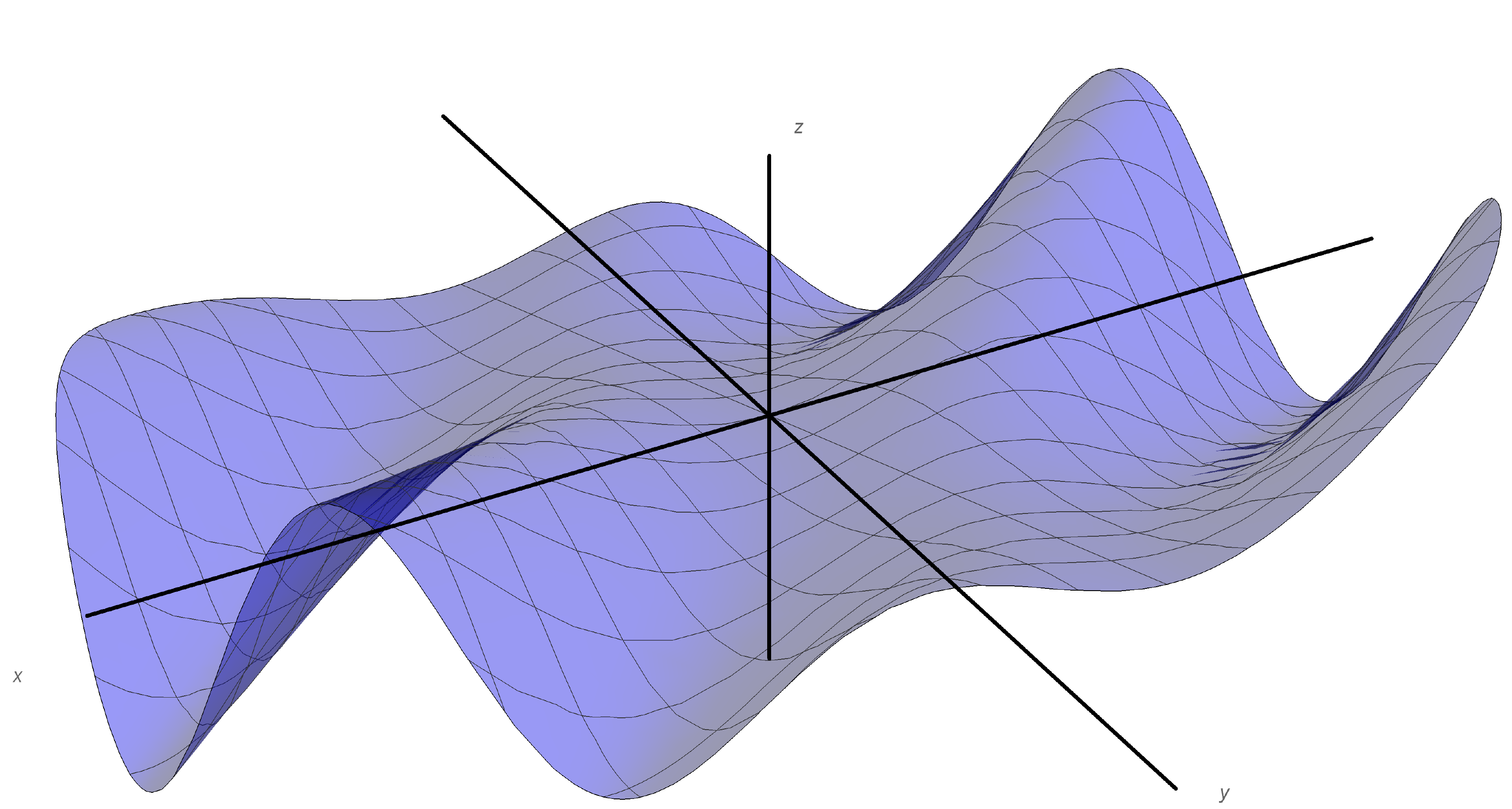}
\caption{Graph of $f(x,y)=x\sin(5x-6y)+y$ on $\mathbf{B}^2$.}
\label{triga}
\end{figure}

\begin{center}
\begin{figure}[h!]
\begin{tabular}{cc}
\begin{minipage}{0.49\textwidth}
\includegraphics[scale=0.26,trim=2cm 9cm 2cm 2cm, clip]{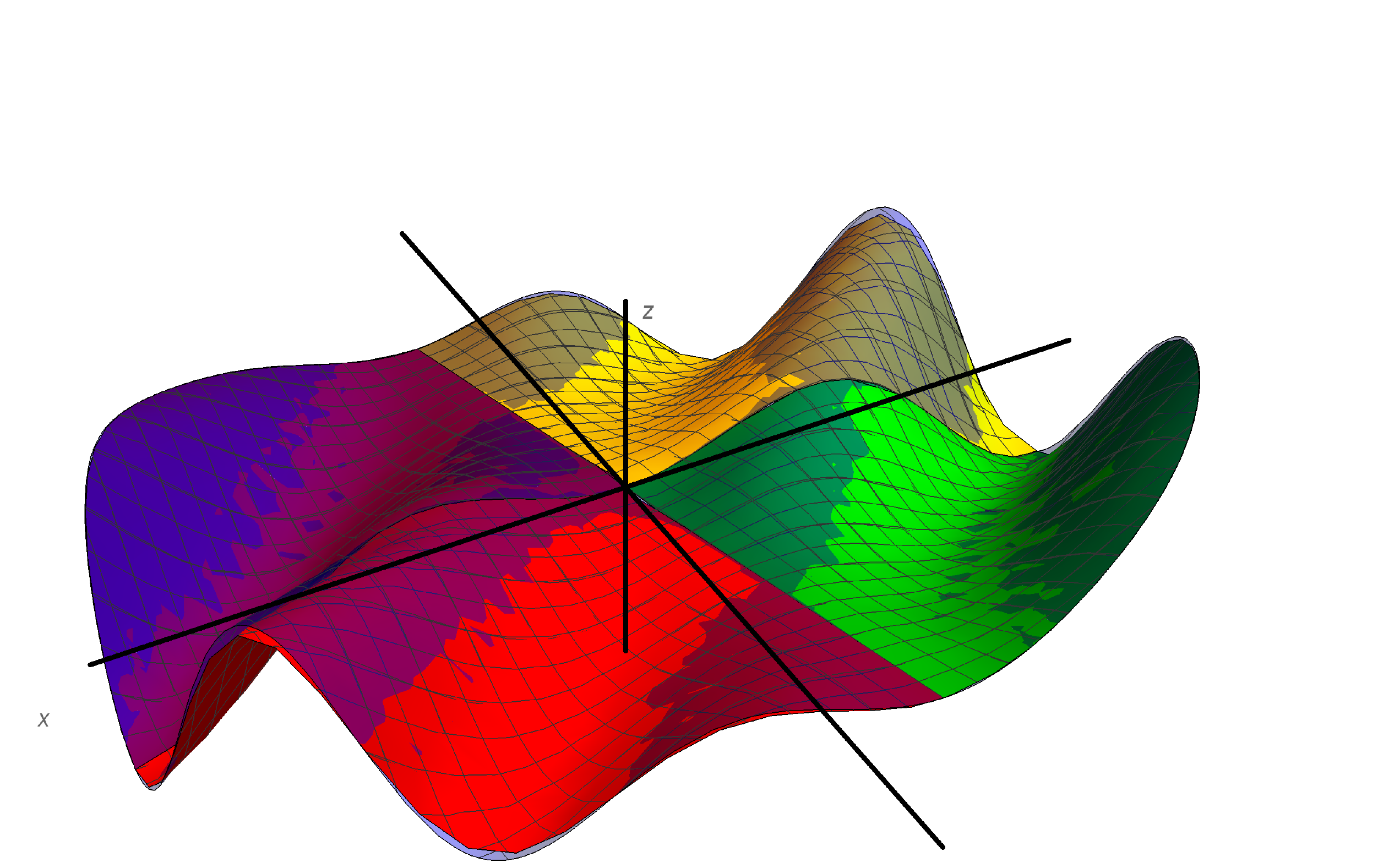} 
\end{minipage}&
\begin{minipage}{0.49\textwidth}
\includegraphics[scale=0.26,trim=17cm 8cm 12cm 2cm, clip]{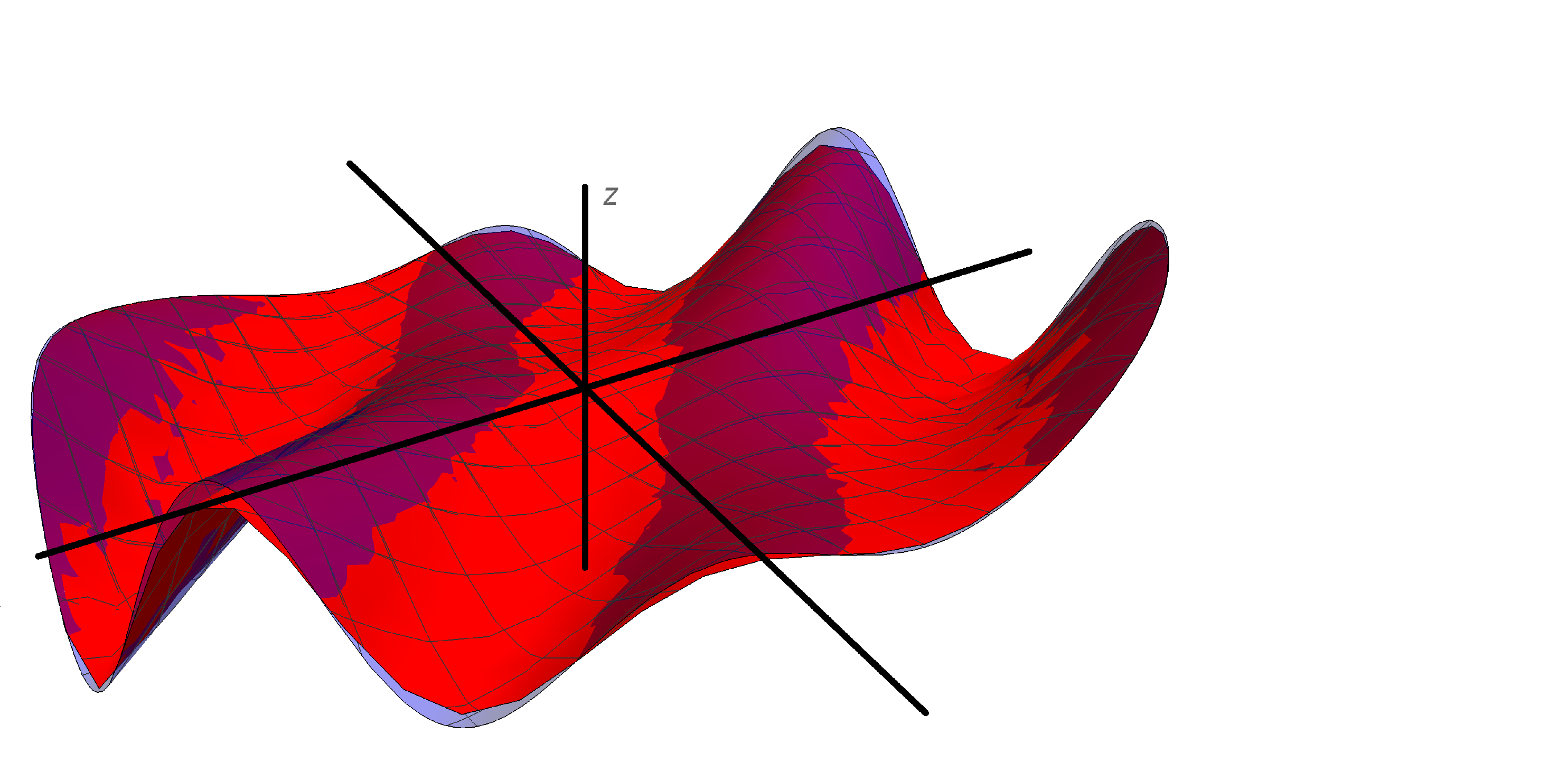} 
\end{minipage}
\end{tabular}
\caption{Approximations overlapped with the graph of $f(x,y)$. Left: $\overline{\mathscr{C}}_n[f(x,y),\mathbf{B}^2]$. Right: $\widetilde{\mathscr{B}}_n[f(x,y),\mathbf{B}^2]$.  }
\label{aprox1}
\end{figure}
\end{center}

\begin{minipage}{\textwidth}

  \begin{minipage}{0.5\textwidth}
    \centering
    $
\begin{array}{|c|c|c|}
\hline &&\\[-0.3cm]
n&\overline{\mathscr{C}}_{n}[f(x,y),\mathbf{B}^2]&\widetilde{\mathscr{B}}_n[f(x,y),\mathbf{B}^2]\\
\hline 
 10 & 0.191411 & 0.30623 \\
 20 & 0.117881 & 0.209091 \\
 30 & 0.0860663 & 0.16182 \\
 40 & 0.0682511 & 0.132416 \\
 50 & 0.0568288 & 0.112151 \\
 60 & 0.0488602 & 0.0972969 \\
 70 & 0.0429694 & 0.0859318 \\
 80 & 0.0384267 & 0.0769527 \\
\hline 
\end{array}
$
    \captionof{table}{RMSE for diffe\-rent values of $n$.}
    \label{trigaterror}
  \end{minipage}
\hspace{1cm}
  \begin{minipage}{0.5\textwidth}
    \centering
\begin{tikzpicture}[scale=0.7]
    \begin{axis}[
axis lines=middle, ymin=0,xmin=0, xlabel=$n$,
        ylabel=RMSE,x label style={at={(1,0)},right=5pt},
  y label style={at={(0,1)},above=5pt},xtick={0,10,...,80},
     ytick={0,0.05,...,0.35},yticklabel style={
        /pgf/number format/fixed,
        /pgf/number format/precision=2
},]
    \addplot[only marks,mark=*,red] plot coordinates {
(10 ,0.191411 )
( 20 , 0.117881 )
( 30 , 0.0860663 )
( 40 , 0.0682511 )
( 50 , 0.0568288 )
( 60 , 0.0488602 )
( 70 , 0.0429694 )
( 80 , 0.0384267 )
    };
    \addlegendentry{$\overline{\mathscr{C}}_{n}[f(x,y),\mathbf{B}^2]$}

    \addplot[only marks,color=blue,mark=o]
        plot coordinates {
( 10,0.30623 )
( 20 ,0.209091 )
( 30 ,0.16182 )
( 40 ,0.132416 )
( 50 ,0.112151 )
( 60 ,0.0972969 )
( 70 ,0.0859318 )
( 80 ,0.0769527 )
        };
    \addlegendentry{$\widetilde{\mathscr{B}}_n[f(x,y),\mathbf{B}^2]$}
    \end{axis}
    \end{tikzpicture}
      
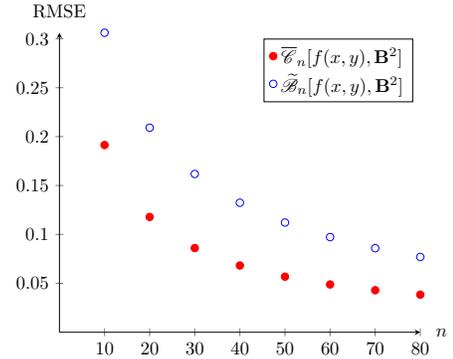
\captionof{figure}{Plot of RMSE in Table \ref{trigaterror}. }
      \label{trigaperror}
    \end{minipage}
  \end{minipage}
  
  \bigskip
  
\subsection{Example 2}

Now, we consider the continuous periodic function 
$$
g(x,y)=\sin(10x+y), \quad (x,y)\in\mathbf{B}^2.
$$
Its graph is shown in Figure \ref{trigb}. It can be observed in Figure \ref{aprox3} that the approximation error for both operators is larger at the maximum and minimum values of the function. Table \ref{trigbterror} and Figure \ref{trigbperror} contain further evidence of this larger error. Moreover, in comparison with the previous example, it seems that the rate convergence of $\overline{\mathscr{C}}_n[g(x,y),\mathbf{B}^2]$ is significantly faster than the rate of convergence of $\widetilde{\mathscr{B}}_n[g(x,y),\mathbf{B}^2]$.

\begin{figure}[h!]
\centering
\includegraphics[scale=0.3,trim=1cm 2cm 1cm 1cm, clip]{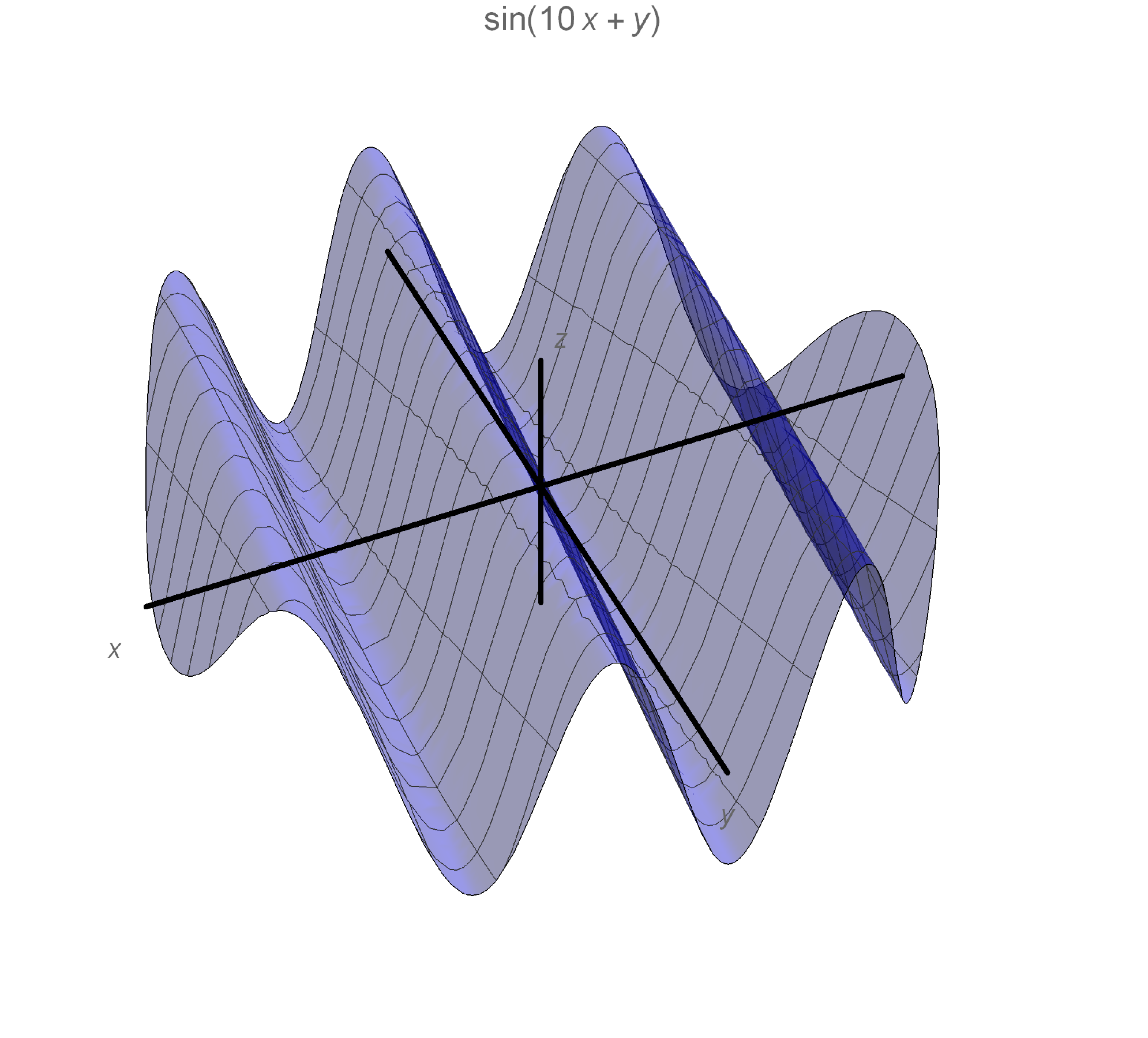}
\caption{Graph of $g(x,y)=\sin(10x+y)$ on $\mathbf{B}^2$.}
\label{trigb}
\end{figure}

\begin{center}
\begin{figure}[h!]
\begin{tabular}{cc}
\begin{minipage}{0.49\textwidth}
\includegraphics[scale=0.13,trim=2cm 9cm 2cm 2cm, clip]{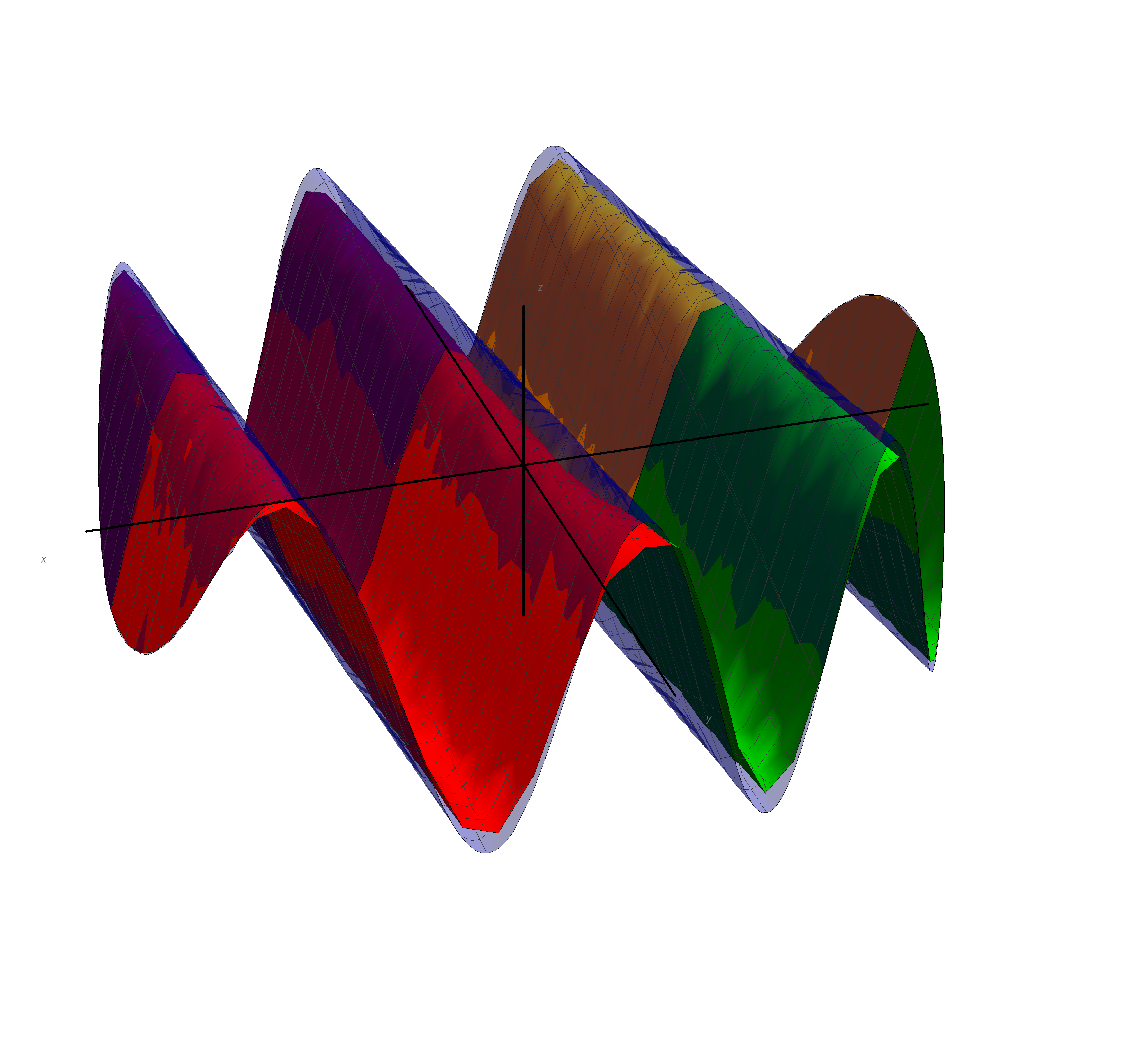} 
\end{minipage}&
\begin{minipage}{0.49\textwidth}
\includegraphics[scale=0.25,trim=8cm 1cm 2cm 0cm, clip]{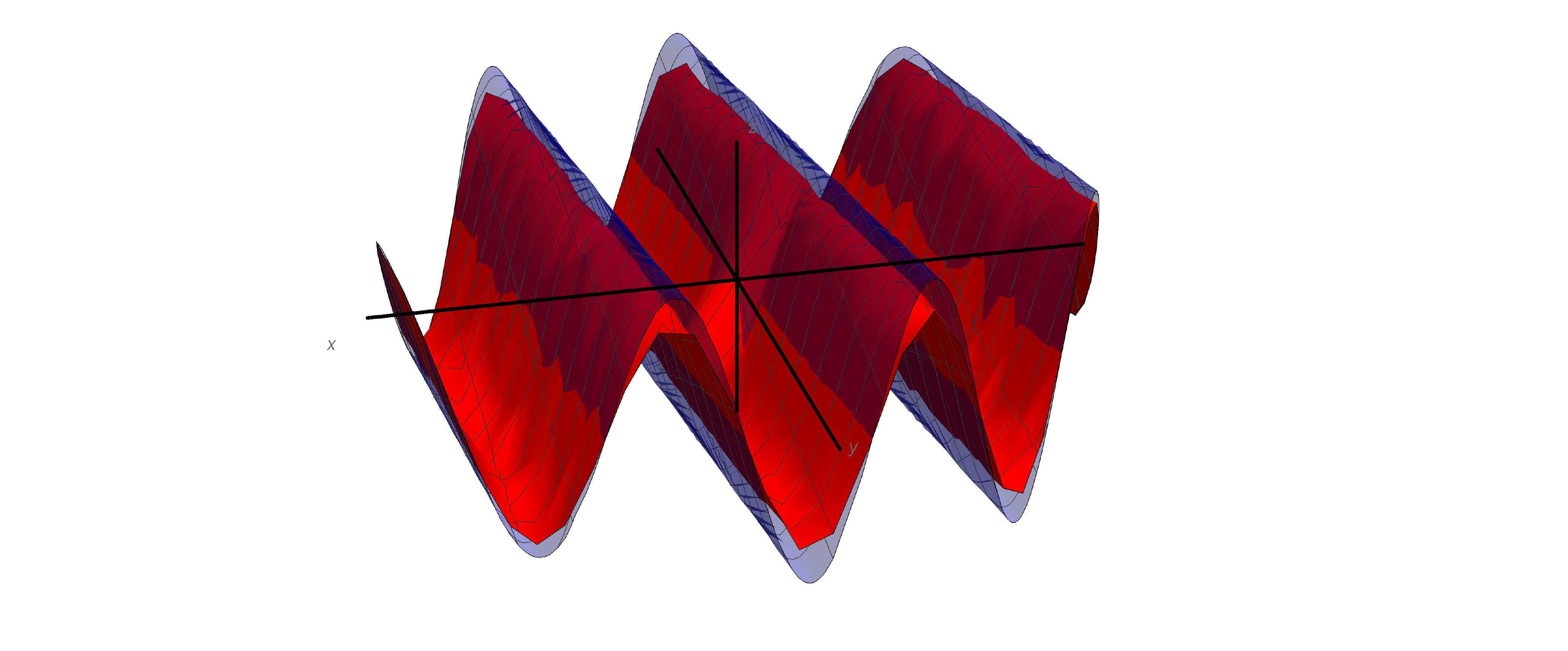} 
\end{minipage}
\end{tabular}
\caption{Approximations overlapped with the graph of $g(x,y)$. Left: $\overline{\mathscr{C}}_n[g(x,y),\mathbf{B}^2]$. Right: $\widetilde{\mathscr{B}}_n[g(x,y),\mathbf{B}^2]$.}
\label{aprox3}
\end{figure}
\end{center}

\begin{minipage}{\textwidth}
\hspace{-2.0cm}
  \begin{minipage}{0.5\textwidth}
    \centering
    $
\begin{array}{|c|c|c|}
\hline &&\\[-0.3cm]
n&\overline{\mathscr{C}}_{n}[g(x,y),\mathbf{B}^2]&\widetilde{\mathscr{B}}_n[g(x,y),\mathbf{B}^2]\\
\hline 
 10 & 0.535344 & 0.700146 \\
 20 & 0.366915 & 0.613427 \\
 30 & 0.278477 & 0.526227 \\
 40 & 0.225091 & 0.454904 \\
 50 & 0.189454 & 0.398559 \\
 60 & 0.163967 & 0.353775 \\
 70 & 0.144812 & 0.317628 \\
 80 & 0.129872 & 0.287968\\
\hline 
\end{array}
$
    \captionof{table}{RMSE for diffe\-rent values of $n$.}
    \label{trigbterror}
  \end{minipage}
\hspace{0cm}
  \begin{minipage}{0.5\textwidth}
    \centering
\begin{tikzpicture}[scale=0.7]
    \begin{axis}[
axis lines=middle, ymin=0,xmin=0, xlabel=$n$,
        ylabel=RMSE,x label style={at={(1,0)},right=5pt},
  y label style={at={(0,1)},above=5pt},xtick={0,10,...,80},
     ytick={0,0.1,...,0.7},yticklabel style={
        /pgf/number format/fixed,
        /pgf/number format/precision=2
},]
    \addplot[only marks,mark=*,red] plot coordinates {
 (10,0.535344)
 (20,0.366915 )
 (30, 0.278477 )
 (40 , 0.225091 )
 (50 , 0.189454 )
 (60 , 0.163967 )
 (70 , 0.144812 )
 (80 , 0.129872)
    };
    \addlegendentry{$\overline{\mathscr{C}}_{n}[g(x,y),\mathbf{B}^2]$}

    \addplot[only marks,color=blue,mark=o]
        plot coordinates {
(10 ,0.700146 )
 (20 ,0.613427 )
 (30, 0.526227 )
 (40 , 0.454904 )
 (50 , 0.398559 )
 (60 , 0.353775 )
 (70 , 0.317628 )
 (80 , 0.287968)
        };
    \addlegendentry{$\widetilde{\mathscr{B}}_n[g(x,y),\mathbf{B}^2]$}
    \end{axis}
    \end{tikzpicture}
      
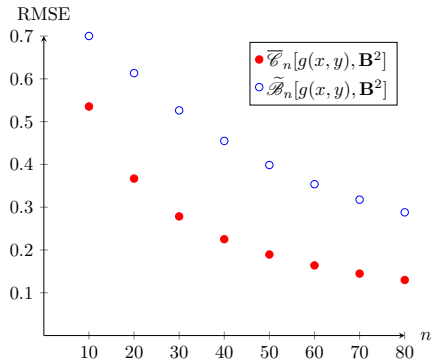
\captionof{figure}{Plot of RMSE in Table \ref{trigbterror}.}
      \label{trigbperror}
    \end{minipage}
  \end{minipage}

\bigskip

\subsection{Example 3}

Here, we consider the continuous function 
$$
h(x,y)=e^{x^2-y^2}-xy, \quad (x,y)\in \mathbf{B}^2,
$$ 
(see Figure \ref{exp}). Both approximations are shown in Figure \ref{aprox4}, and their respective RSME are listed in Table \ref{expterror} and plotted in Figure \ref{expperror}. Observe that, in this case, the RSME for both approximations are significantly smaller than in the previous examples. Moreover, based on Figure \ref{expperror}, it seems that for sufficiently large values of $n$, the rate of convergence of both approximations is considerably similar to each other.

\begin{figure}[h!]
\centering
\includegraphics[scale=0.3,trim=3cm 3cm 4cm 1cm, clip]{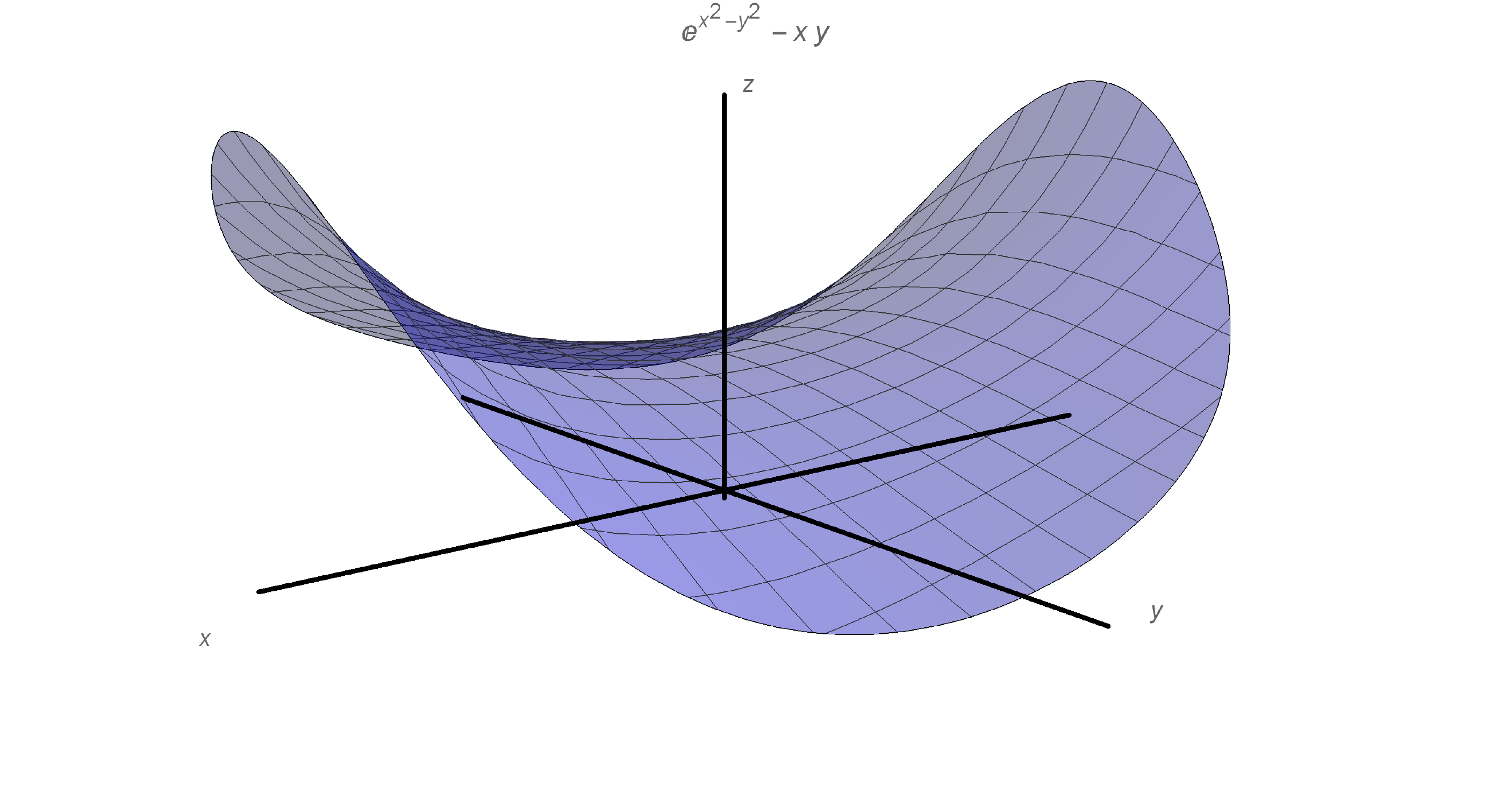}
\caption{Graph of $h(x,y)=e^{x^2-y^2}-xy$.}
\label{exp}
\end{figure}

\begin{figure}[h!]
\centering
\begin{tabular}{cc}
\begin{minipage}{0.49\textwidth}
\includegraphics[scale=0.22,trim=5cm 0cm 5cm 0cm, clip]{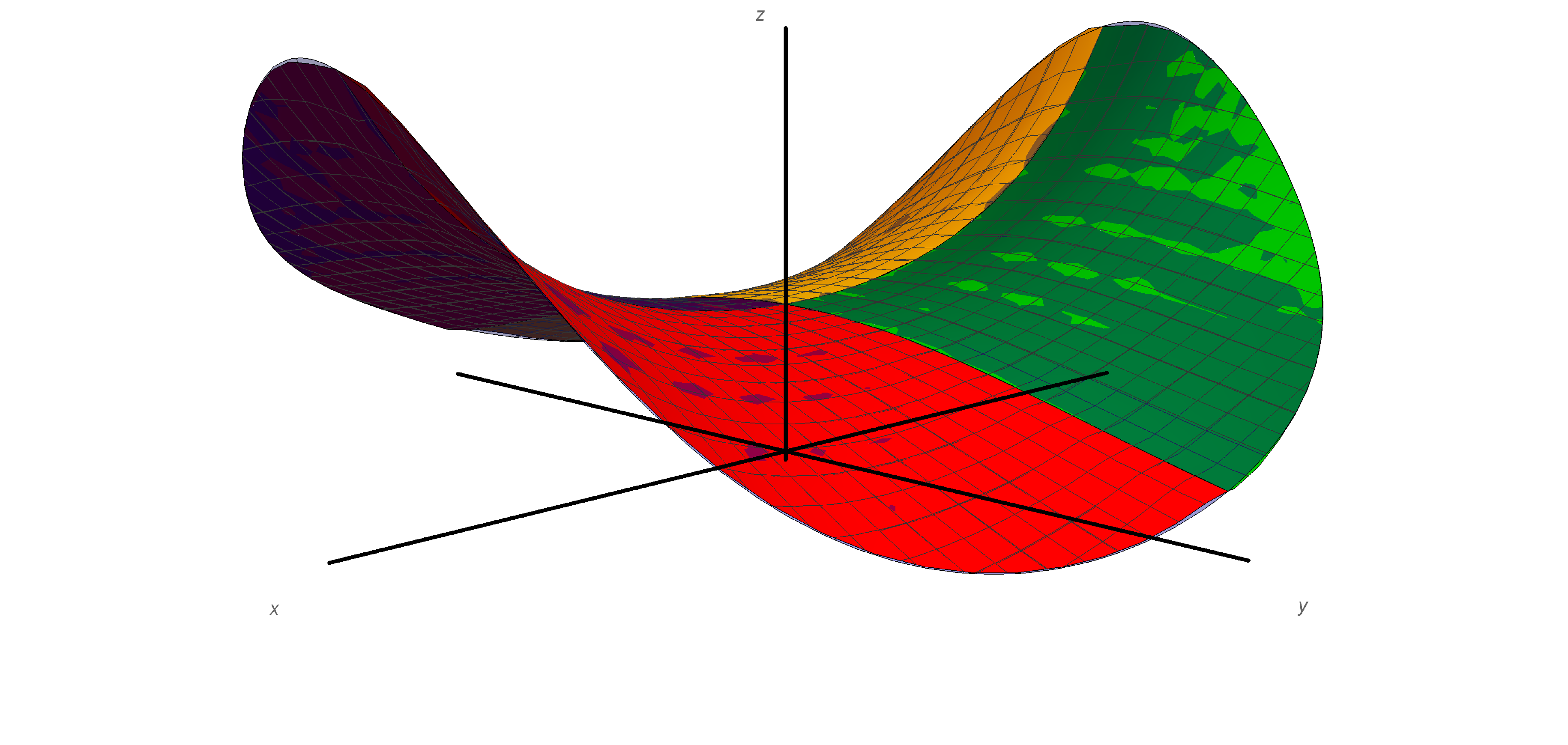} 
\end{minipage}&
\begin{minipage}{0.5\textwidth}
\includegraphics[scale=0.45,trim=2cm 0cm 2cm 0cm, clip]{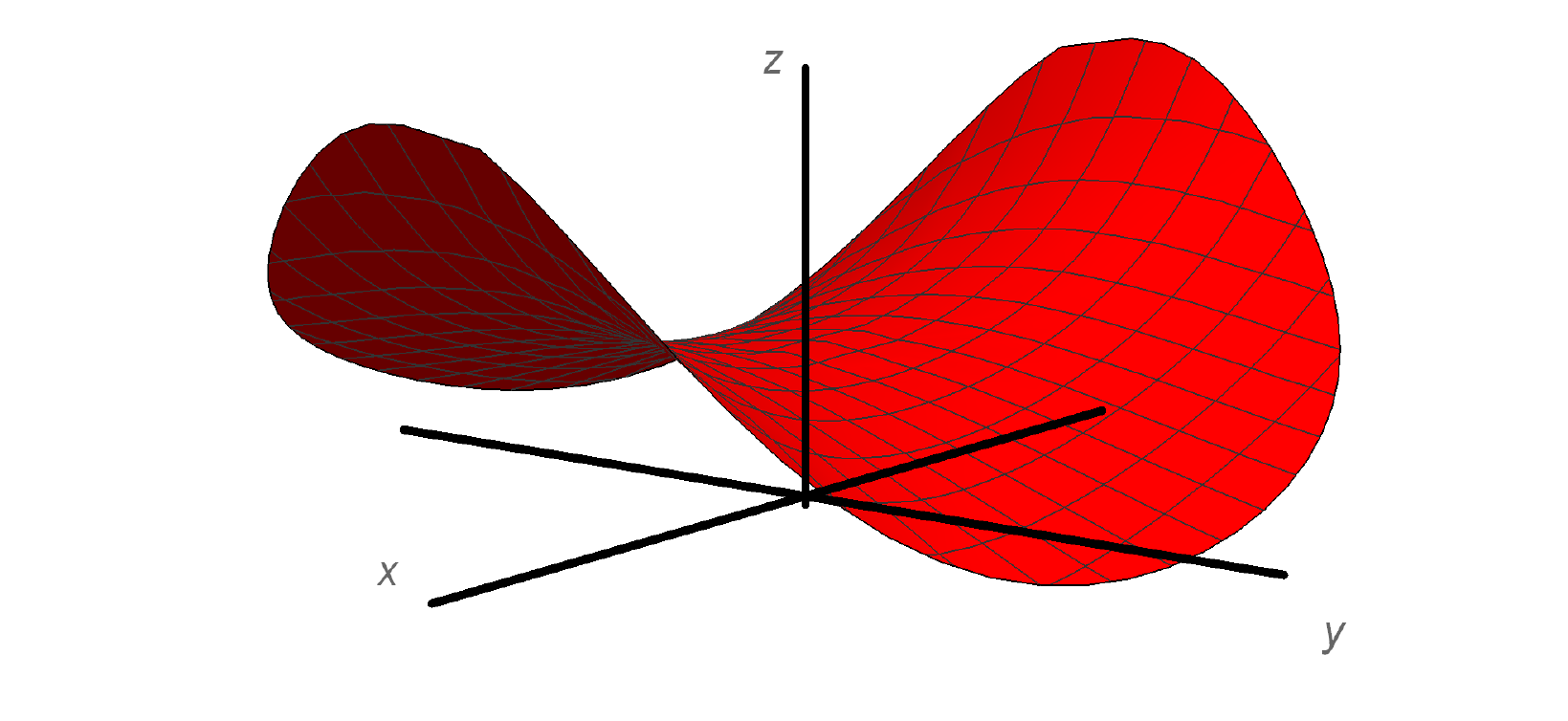} 
\end{minipage}
\end{tabular}
\caption{Approximations overlapped with the graph of $h(x,y)$. Left: $\overline{\mathscr{C}}_n[h(x,y),\mathbf{B}^2]$. Right: $\widetilde{\mathscr{B}}_n[h(x,y),\mathbf{B}^2]$. }
\label{aprox4}
\end{figure}

\bigskip

\begin{minipage}{\textwidth}
\hspace{-1.0cm}
  \begin{minipage}{0.5\textwidth}
    \centering
$
\begin{array}{|c|c|c|}
\hline &&\\[-0.3cm]
n&\overline{\mathscr{C}}_{n}[h(x,y),\mathbf{B}^2]&\widetilde{\mathscr{B}}_n[h(x,y),\mathbf{B}^2]\\
\hline  
 10 & 0.0505862 & 0.140837 \\
 20 & 0.0293585 & 0.0685387 \\
 30 & 0.0213945 & 0.0455634 \\
 40 & 0.017105 & 0.0342737 \\
 50 & 0.0143844 & 0.0275514 \\
 60 & 0.0124871 & 0.0230843 \\
 70 & 0.0110789 & 0.0198962 \\
 80 & 0.00998647 & 0.0175041 \\
\hline 
\end{array}
$
    \captionof{table}{RMSE for diffe\-rent values of $n$.}
    \label{expterror}
  \end{minipage}
\hspace{0cm}
  \begin{minipage}{0.5\textwidth}
    \centering
\begin{tikzpicture}[scale=0.7]
    \begin{axis}[
axis lines=middle, ymin=0,xmin=0, xlabel=$n$,
        ylabel=RMSE,x label style={at={(1,0)},right=5pt},
  y label style={at={(0,1)},above=5pt},xtick={0,10,...,80},
     ytick={0,0.02,...,1.20},yticklabel style={
        /pgf/number format/fixed,
        /pgf/number format/precision=2
},]
    \addplot[only marks,mark=*,red] plot coordinates {
 (10 , 0.0505862 )
 (20 , 0.0293585 )
 (30 , 0.0213945 )
 (40 , 0.017105 )
 (50 , 0.0143844 )
 (60 , 0.0124871 )
 (70 , 0.0110789 )
 (80 , 0.00998647 )
    };
    \addlegendentry{$\overline{\mathscr{C}}_{n}[h(x,y),\mathbf{B}^2]$}

    \addplot[only marks,color=blue,mark=o]
        plot coordinates {
(10,0.140837)
( 20 , 0.0685387 )
 (30 , 0.0455634 )
( 40 ,0.0342737 )
( 50 , 0.0275514 )
( 60 , 0.0230843 )
( 70 , 0.0198962 )
( 80 , 0.0175041 )
        };
    \addlegendentry{$\widetilde{\mathscr{B}}_n[h(x,y),\mathbf{B}^2]$}
    \end{axis}
    \end{tikzpicture}
      
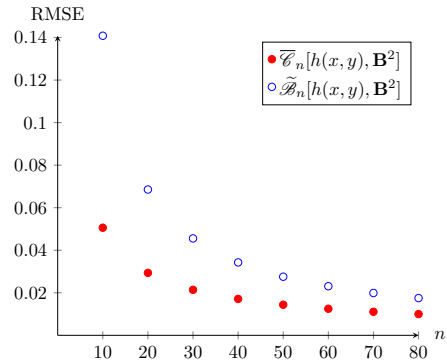
\captionof{figure}{Plot of RSME in Table \ref{expterror}.}
      \label{expperror}
    \end{minipage}
  \end{minipage}

\bigskip

\subsection{Example 4}

In this numerical example, we are interested in observing the behavior of Bernstein-type and Bernstein-Stancu operators at jump discontinuities.

Let us consider the following discontinuous function:
\begin{equation*}
\eta(x,y)=\begin{cases}
1,&\text{if\ \ } x^2+y^2<0.5,\\
0,&\text{if\ \ } 0.5\leqslant x^2+y^2\leqslant 0.8,\\
0.5,&\text{if\ \ } 0.8<x^2+y^2\leqslant 1.\\
\end{cases}\end{equation*}
The graph of $\eta(x,y)$ is shown in Figure \ref{piece} and the approximations are shown in Figure \ref{aprox4}. It is interesting to observe the behavior of the approximations at the points of jump discontinuities and, thus, we have included Figure \ref{gibbs}, where we show a cross sectional view of the approximations with increasing values of $n$. As in the univariate case, it seems that the Gibbs phenomenon does not occur. Finally, Table \ref{pieceterror} and Figure \ref{pieceperror} expose a significantly slow convergence rate for this discontinuous function in comparison with the previous continuous examples.

\begin{figure}[h!]
\centering
\includegraphics[scale=0.25,trim=3cm 3cm 4cm 4cm, clip]{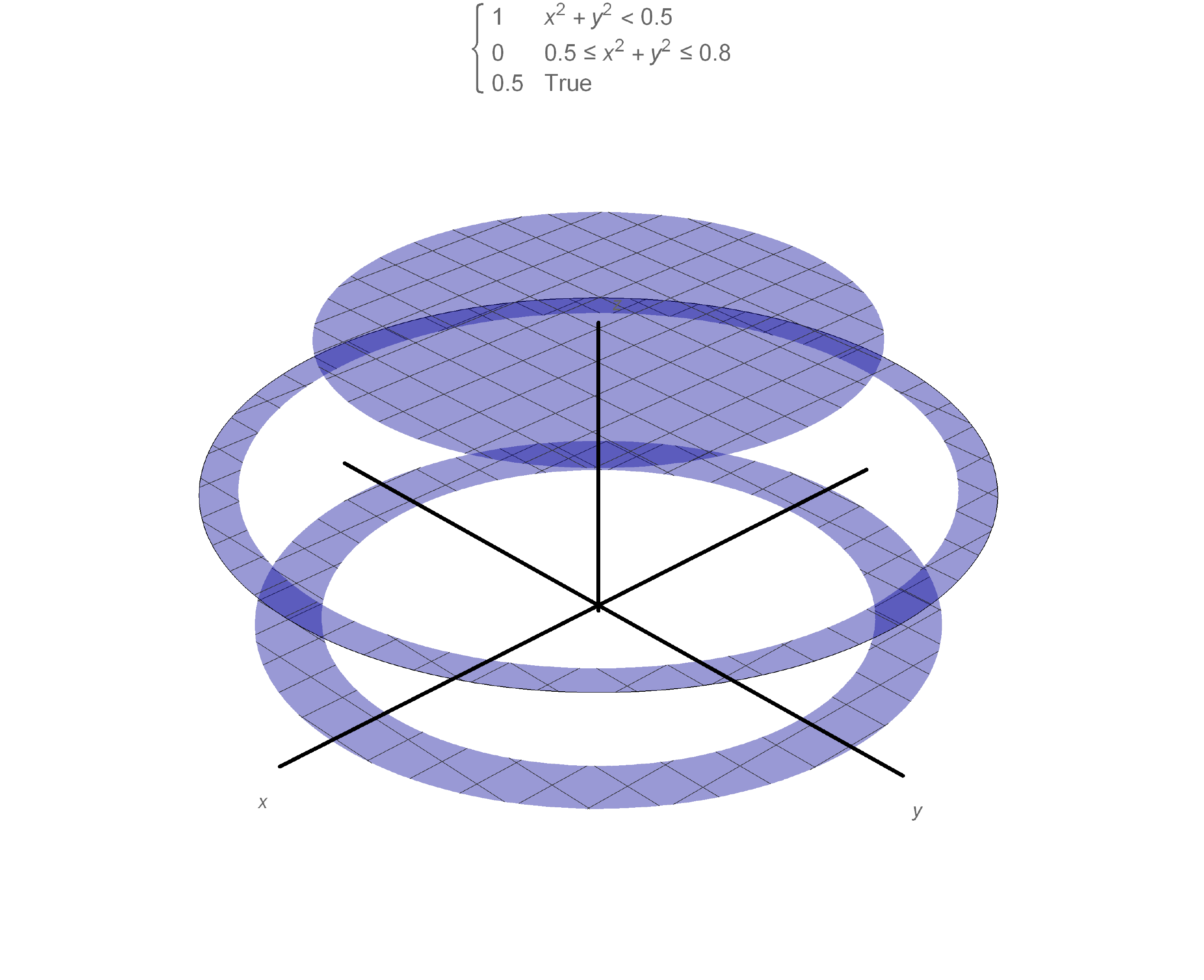}
\caption{Graph of $\eta(x,y)$.}
\label{piece}
\end{figure}

\begin{center}
\begin{figure}[h!]
\begin{tabular}{cc}
\begin{minipage}{0.5\textwidth}
\includegraphics[scale=0.21,trim=5cm 5cm 5cm 0cm, clip]{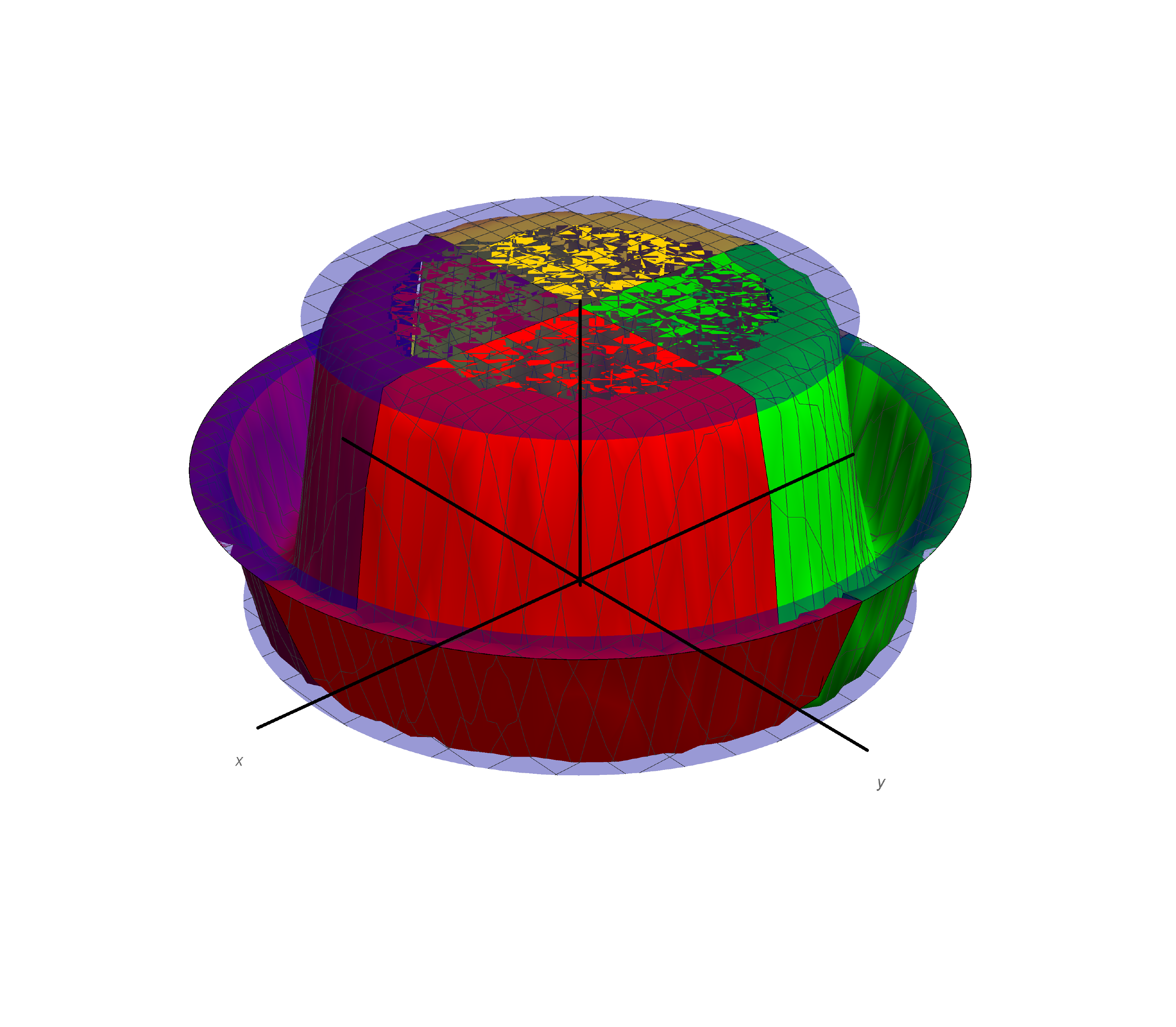} 
\end{minipage}&
\begin{minipage}{0.5\textwidth}
\includegraphics[scale=0.30,trim=12cm 3cm 12cm 0cm, clip]{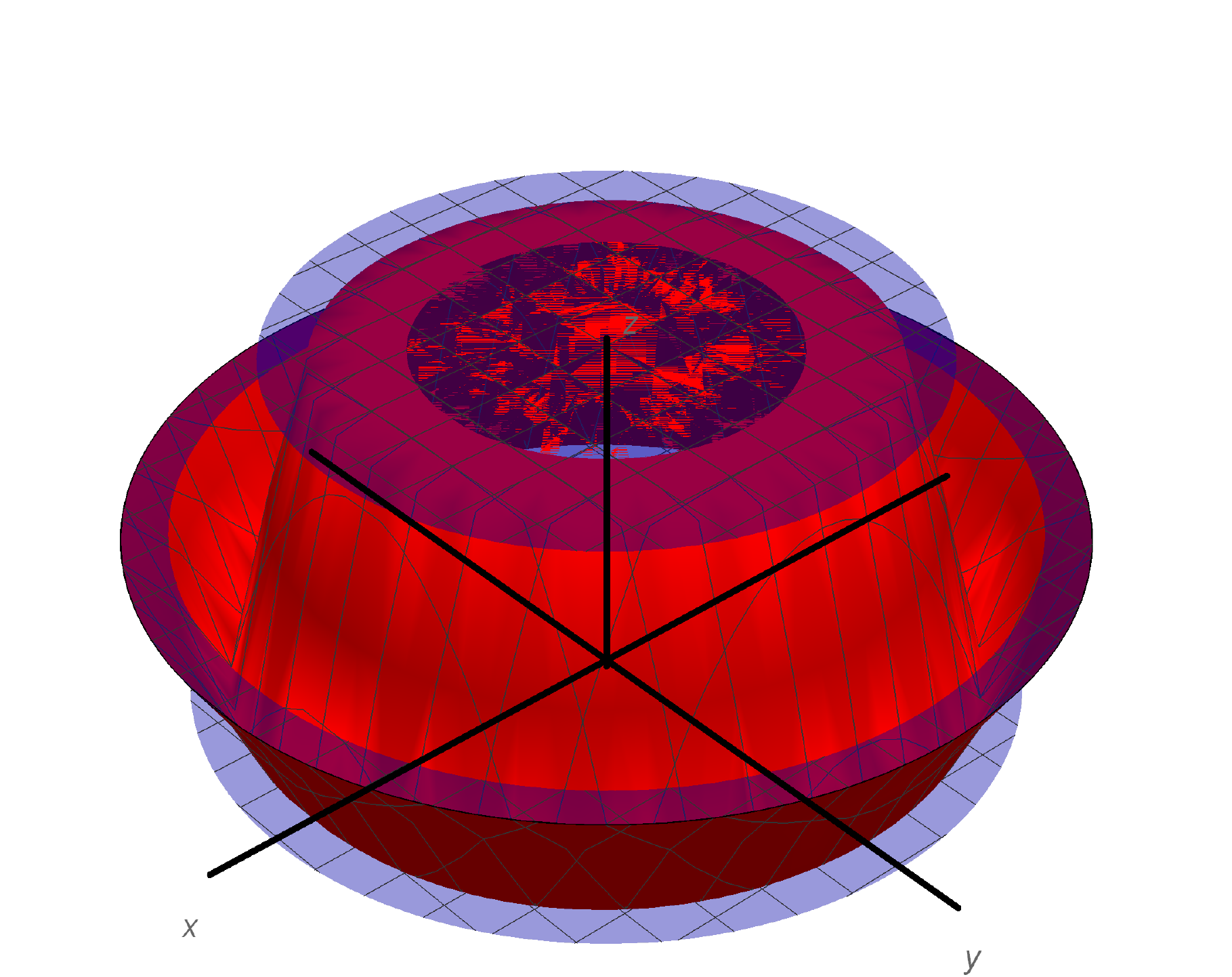} 
\end{minipage}
\end{tabular}
\caption{Approximations overlapped with the graph of $\eta (x,y)$. Left: $\overline{\mathscr{C}}_n[\eta(x,y),\mathbf{B}^2]$. Right: $\widetilde{\mathscr{B}}_n[\eta(x,y),\mathbf{B}^2]$. }
\label{aprox5}
\end{figure}
\end{center}

\begin{center}
\begin{figure}[h!]
\begin{tabular}{cc}
\begin{minipage}{0.49\textwidth}
\includegraphics[scale=0.13,trim=0cm 0cm 0cm 0cm, clip]{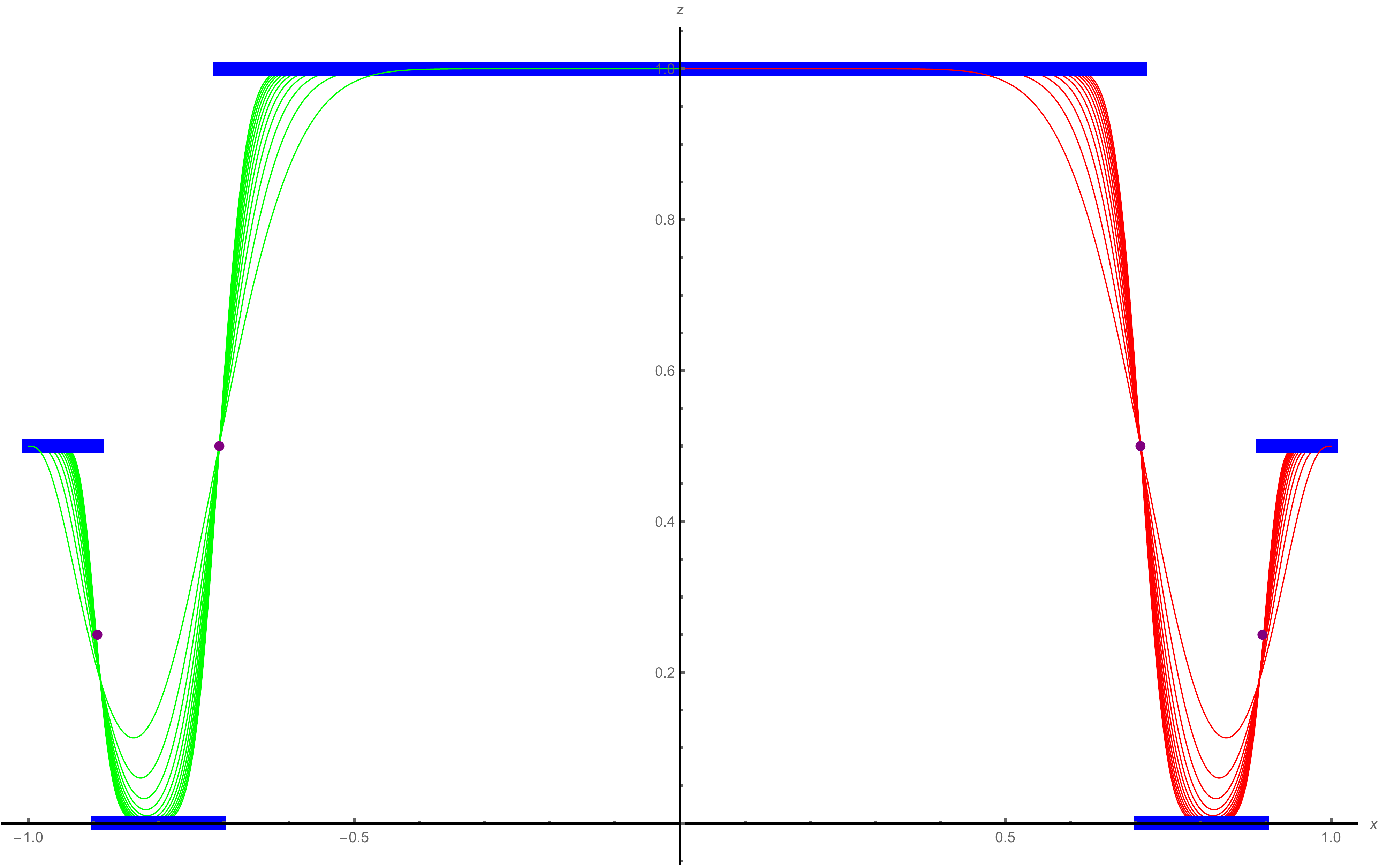} 
\end{minipage}&
\begin{minipage}{0.49\textwidth}
\includegraphics[scale=0.13,trim=0cm 0cm 0cm 0cm, clip]{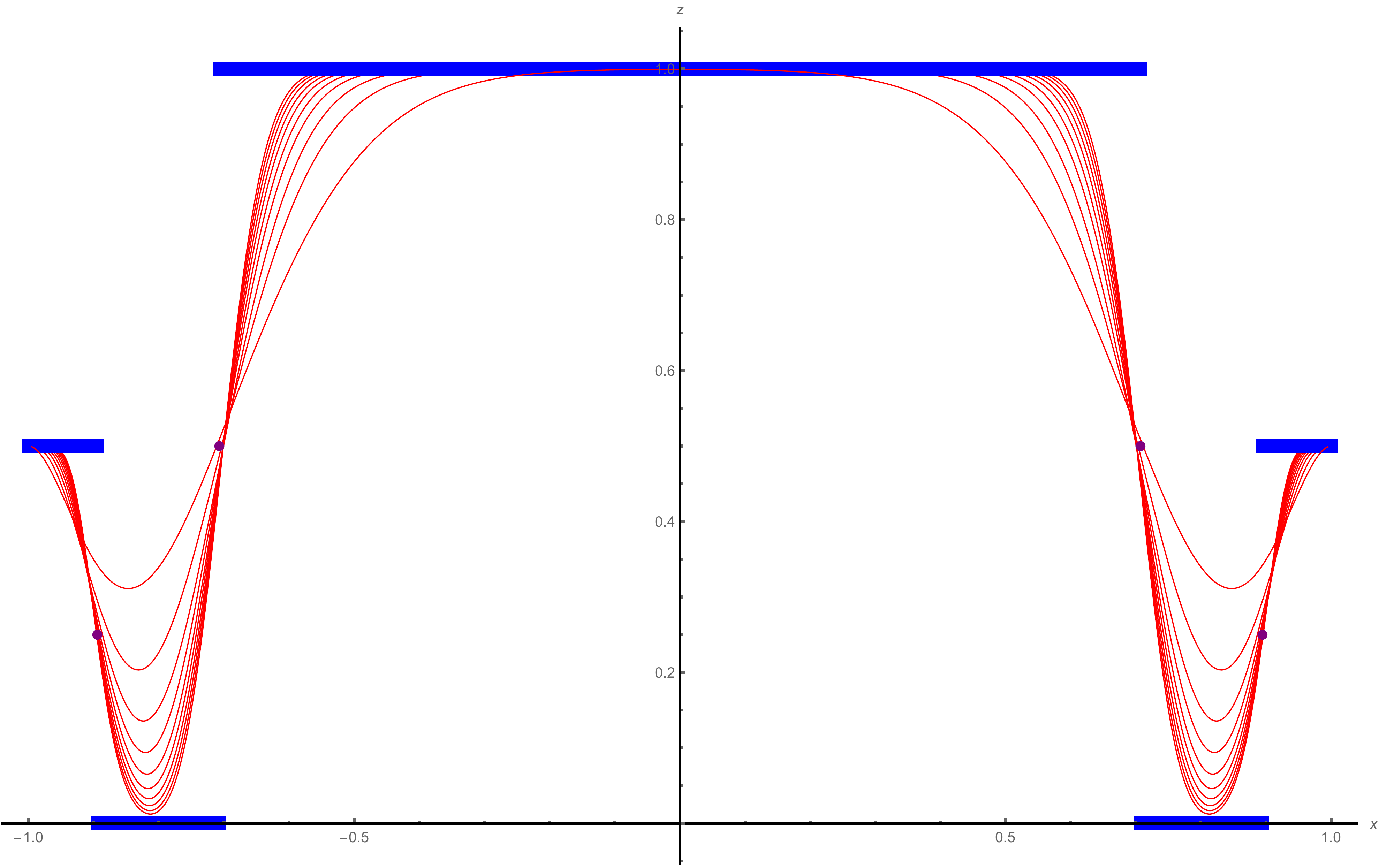} 
\end{minipage}
\end{tabular}
\caption{Cross sectional view of the approximations for increasing values of $n$. Left: $\overline{\mathscr{C}}_n[\eta(x,y),\mathbf{B}^2]$. Right: $\widetilde{\mathscr{B}}_n[\eta(x,y),\mathbf{B}^2]$. }
\label{gibbs}
\end{figure}
\end{center}

\newpage
\bigskip
\begin{minipage}{\textwidth}
\hspace{-1.0cm}
  \begin{minipage}{0.5\textwidth}
    \centering
$
\begin{array}{|c|c|c|}
\hline &&\\[-0.3cm]
n&\overline{\mathscr{C}}_n[\eta(x,y),\mathbf{B}^2]&\widetilde{\mathscr{B}}_n[\eta(x,y),\mathbf{B}^2]\\
\hline  10 & 0.216588 & 0.270366 \\
 20 & 0.175754 & 0.243468 \\
 30 & 0.156563 & 0.223305 \\
 40 & 0.144805 & 0.210916 \\
 50 & 0.136559 & 0.205949 \\
 60 & 0.130305 & 0.192988 \\
 70 & 0.125319 & 0.193887 \\
 80 & 0.121205 & 0.187675 \\
\hline 
\end{array}
$
    \captionof{table}{RMSE for different values of $n$.}
    \label{pieceterror}
  \end{minipage}
\hspace{0cm}
  \begin{minipage}{0.5\textwidth}
    \centering
\begin{tikzpicture}[scale=0.7]
    \begin{axis}[
axis lines=middle, ymin=0.05,xmin=0, xlabel=$n$,
        ylabel=RMSE,x label style={at={(1,0)},right=5pt},
  y label style={at={(0,1)},above=5pt},xtick={0,10,...,80},
     ytick={0,0.05,...,0.35},yticklabel style={
        /pgf/number format/fixed,
        /pgf/number format/precision=2
},]
    \addplot[only marks,mark=*,red] plot coordinates {
(10 ,0.216588)
( 20 , 0.175754)
( 30 , 0.156563)
( 40 , 0.144805)
( 50 , 0.136559 )
( 60 , 0.130305 )
( 70 , 0.125319 )
( 80 , 0.121205 )
    };
    \addlegendentry{$\overline{\mathscr{C}}_{n}[\eta(x,y),\mathbf{B}^2]$}

    \addplot[only marks,color=blue,mark=o]
        plot coordinates {
(10 , 0.270366 )
( 20 , 0.243468 )
 (30 , 0.223305 )
 (40 , 0.210916 )
( 50 , 0.205949 )
( 60 , 0.192988 )
( 70 , 0.193887 )
( 80 , 0.187675 )
        };
    \addlegendentry{$\widetilde{\mathscr{B}}_n[\eta(x,y),\mathbf{B}^2]$}
    \end{axis}
    \end{tikzpicture}
      
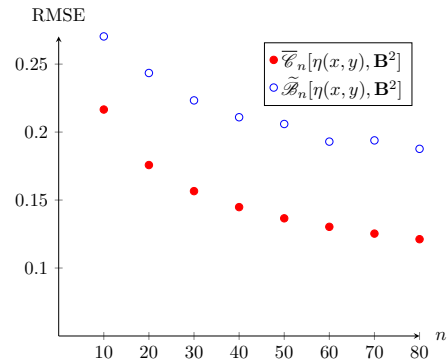
\captionof{figure}{Plot of RSME in Table \ref{pieceterror}.}
      \label{pieceperror}
    \end{minipage}
  \end{minipage}

\newpage



\begin{thebibliography}{15}

\bibitem{BeScXu92}  H. Berens, H. J. Schmid, Y. Xu,
	\textit{Bernstein–Durrmeyer polynomials on a simplex},
    J. Approx. Theory 68 (1992),  no 3,  247-261.
   
\bibitem{Be12} S. Bernstein,
	\textit{Demonstration di Th\'eor\`eme de Weierstrass fond\'ee sur le calcul des probabilit\'es},
	Commun. Soc. Math. Kharkov (2), 13 (1912-1913), 1,2.
    Communications of the Kharkov Mathematical Society, Volume XIII, 1912/13 (p 1-2).
    
\bibitem{CGM06} D. C\'ardenas-Morales, P. Garrancho, F. J. Mu\~noz-Delgado, 
    \textit{Shape preserving approximation by Bernstein-type operators which fix polynomials}, 
    Appl. Math. Comput. 182 (2006) 1615–1622.

\bibitem{CGR11} D. C\'ardenas-Morales, P. Garrancho, I. Ra\c{s}a,
    \textit{Bernstein-type operators that preserve polynomials},
    Comput. Math. with Appl. 62 (2011), 158--163.
  
\bibitem{De81} M. M. Derriennic, 
	\textit{Sur l'approximation de fonctions int\'egrables sur $[0,1]$ par des polyn\^omes de Bernstein modifi\'es}
	J. Approx. Theory 31 (1981), no. 4, 325-343.
	
\bibitem{De85} M. M. Derriennic, 
	\textit{On multivariate approximation by Bernstein-type polynomials},
	J. Approx. Theory 45 (1985), no. 2, 155-166. 

\bibitem{DX14}
C. F. Dunkl, Y. Xu,
    \textit{Orthogonal polynomials of several variables},
    2nd edition, Encyclopedia of Mathematics and its Applications, vol. 155, Cambridge Univ. Press, Cambridge (2014).
	
\bibitem{Du67} J. L. Durrmeyer,
	\textit{Une formule d'inversion de la transform\'ee de Laplace. Applications \`a la th\'eorie des moments},
	Th\`ese 3e cycle, Fac. des Sciences de l'Universit\'e de Paris, 1967.	
	
\bibitem{GPR09}	H. Gonska, P. Pi\c{t}ul, I. Ra\c{s}a, 
     \textit{General King-type operators}, 
     Results Math. 53 (3–4) (2009) 279–286.
	
\bibitem{Ka30} L. V. Kantorovitch,
	\textit{Sur certains d\'evelopements suivants les polyn\^omes de la forme de S. Bernstein},
	I, II, C. R. Acad. Sci. USSR (1930), 563-568, 595-600.
	
\bibitem{Ki03}	J. P. King, 
    \textit{Positive linear operators which preserve $x^2$}, 
     Acta Math. Hungar. 99 (3) (2003) 203–208.

\bibitem{L97} G. G. Lorentz,
Bernstein  polynomials,
Chelsea Publishing Company, New York, 1997.


\bibitem{Sa1} P. Sablonnière, 
\textit{Opérateurs de Bernstein-Jacobi  et polynômes orthogonaux}, Publ. ANO 37, Laboratoire de Calcul, Université de Lille, (1981).

\bibitem{Schurer63} 
F. Schurer,
\textit{On the approximation of functions of many variables with linear positive operators},
Nederl. Akad. Wetensch. Proc. Ser. A 66 = Indag. Math. \textbf{25} (1963), 313--327.

\bibitem{Stan63} D. D. Stancu,
    \textit{A Method for obtaining polynomials of Bernstein type of two variables}, {The American Mathematical Monthly}, 70 (1963), no 3, {260-264}.

\end{thebibliography}
\end{document}